\documentclass[10pt,reqno]{amsart}

\usepackage{amsfonts,amsmath,amssymb,amsthm}
\usepackage{tikz}
\usepackage[margin=1in]{geometry}
\usepackage{enumerate, enumitem, multirow, array}
\usepackage{makecell} 
\usepackage{changepage}
\usepackage{threeparttable}
\usepackage{tikz}

\usepackage[colorlinks=true,citecolor=black,linkcolor=black,urlcolor=blue]{hyperref}

\theoremstyle{plain}
\newtheorem{thm}{Theorem}[section]
\newtheorem{lemma}[thm]{Lemma}
\newtheorem{prop}[thm]{Proposition}
\newtheorem{obs}[thm]{Observation}
\newtheorem{claim}[thm]{Claim}
\newtheorem{cor}[thm]{Corollary}
\newtheorem{conj}[thm]{Conjecture}

\theoremstyle{definition}
\newtheorem{defn}[thm]{Definition}

\theoremstyle{remark}
\newtheorem*{rem}{Remark}

\usepackage[foot]{amsaddr}

\title{On different ``middle parts'' of a tree}

\author{Heather Smith$^{*}$}
\author{L\'aszl\'o Sz\'ekely$^{\dagger}$}
\author{Hua Wang$^{\ddagger}$}
\author{Shuai Yuan$^\dagger$}

\address{$^*$ School of Mathematics, Georgia Institute of Technology ,Atlanta, Georgia, USA}

\address{$^\dagger$ Department of Mathematics, University of South Carolina, Columbia, South Carolina, USA}

\address{$^\ddagger$ Department of Mathematical Sciences, Georgia Southern University, Statesboro, Georgia, USA}

\thanks{Emails: {\tt heather.smith@math.gatech.edu}, {\tt \{szekely,syuan\}@math.sc.edu}, {\tt hwang@georgiasouthern.edu}}

\begin{document}

\begin{abstract}
We determine the maximum distance between any two of the center, centroid, and subtree core among trees with a given order. Corresponding results are obtained for trees with given maximum degree and also for trees with given diameter. The problem of the maximum distance between the
centroid and the subtree core among trees with given order and diameter becomes difficult. It can be solved in terms of the problem of minimizing the number of root-containing subtrees in a rooted tree of given order and height.  While the latter problem remains unsolved,  we provide a partial characterization of the extremal structure.

\end{abstract}

\maketitle


\section{Preliminaries}
Many real-valued functions defined on the vertex set of a tree have been studied in the literature. Such functions are closely related to graph invariants motivated from practical applications (such as the Wiener index in biochemistry \cite{wiener}) or pure mathematical study (such as the subtree density \cite{jamison}). In this paper we are interested in the {\it eccentricity} and {\it distance of a vertex},
as well as the {\it number of subtrees containing a vertex}. Both eccentricity and distance of a vertex are minimized at one vertex or two adjacent vertices. These vertices define the {\it center} and {\it centroid} respectively
of the tree. The number of subtrees containing a vertex is maximized by one vertex or two adjacent vertices, called the {\it subtree core} of the tree.

The study of the center and centroid (Definitions \ref{def: C} and \ref{def: CT}) can be traced back to \cite{jordan}. This paper focuses on the three different middle parts of the tree: the center, the centroid, and the subtree core. We investigate the geometry of their possible locations, in particular the extremal problem of  how far they can be from each other.
The resulting three problems are solved in Section~\ref{dist middle parts} for arbitrary trees of fixed order, and  the corresponding extremal structures are also characterized. In Section~\ref{deg tree}, 
the same three problems are solved  for degree bounded trees and conjectures are proposed  for the case of  binary trees of fixed order. In Section \ref{diam tree}, we consider the same questions for trees with fixed order and bounded diameter. In this setting, the maximum distance between the center and the centroid and between the center and the subtree core are determined, but the problem of the maximum distance between the
centroid and the subtree core becomes challenging. It can be solved in terms of the problem of minimizing the number of root-containing subtrees in a rooted tree of given order and height.
While this latter problem remains unsolved, in Section~\ref{sec: Min_subtrees} we provide a partial characterization of the extremal structure.

Let us begin with formal definitions. The distance in the tree from $u$ to $v$, denoted $d(u,v)$, is the number of edges on their unique connecting path $P(u,v)$.

\begin{defn}
The {\em eccentricity} of a vertex $v$ in a tree $T$ is \[ecc_T(v)=\max_{u \in V(T)}  d(v,u) .\] The {\em center} of $T$, denoted $C(T)$, is the set of vertices which have the minimum eccentricity among all vertices in the tree.
\label{def: C}
\end{defn}

\begin{defn}
The {\em distance} of a vertex $v$ in a tree $T$, denoted $d_T(v)$, is  \[d_T(v)=\sum_{u \in V(T)} d(v,u).\] The {\em centroid} of $T$, denoted $CT(T)$, is the set of vertices which have the minimum distance among all vertices in the tree.
\label{def: CT}
\end{defn}

A \emph{subtree} of tree $T$ is a connected subgraph which is induced on a nonempty set of vertices. We consider $T$ to be a subtree of itself and a single vertex is also a subtree of $T$. 
\begin{defn}
As the name suggests, the \textit{number of subtrees} of a vertex $v$ in a tree $T$, denoted $F_T(v)$, is the number of subtrees of $T$ which contain $v$.  The {\em subtree core} of a tree $T$, denoted $Core(T)$, is the set of vertices that maximize the function $F_T(.)$ \cite{laszlo}.
\end{defn}

If $H$ is a forest and $v$ is a vertex in $H$, then $F_H(v)$ will be defined, as above, to be the number of subtrees of $H$ which contain vertex $v$. In particular, all subtrees which are counted must be subtrees of the component of $H$ which contains vertex $v$. 

Jordan \cite{jordan} found that $C(T)$ consists of either one vertex or two adjacent vertices (see also Ex. 6.21a in \cite{lovasz}).
Given the vertices along any path of a tree, 
the sequence of $F_T(.)$ function values is strictly concave down (\cite{laszlo}), 
the sequence of $d(.)$ function values are strictly concave up (Ex. 6.22 in \cite{lovasz}; \cite{entringer}), and 
the sequence of $ecc_T(.)$ function values are  concave up (Ex. 6.21 in \cite{lovasz}).
Strict concavity immediately implies that the sets $CT(T)$ and $Core(T)$ either consist of one vertex or two adjacent vertices. 

We are specifically interested in how the middle sets are located, relative to one another. It is well-known that $C(T)$ and $CT(T)$ can be far apart 
(Ex. 6.22c in \cite{lovasz}), and that $Core(T)$ can differ from them \cite{laszlo}. 

There are some natural questions that we will explore. How close to each other can they be? How far apart can they be spread? Must they lie on a common path? Can they appear in any ordering? 

It is easy to find trees where $C(T)$, $CT(T)$, and $Core(T)$ coincide, like the star and paths of even length to name a few. It is more interesting to see how far apart these middle sets can be in a single tree.

Considering one vertex from each of $C(T)$, $CT(T)$, and $Core(T)$, any two of these must lie on a common path.
However, it is possible that the vertices from $C(T)$, $CT(T)$, and $Core(T)$ in the same tree $T$ do not all lie on a common path. Figure~\ref{fig:ex1} provides an example 
of this very situation.

\begin{figure}[htbp]
\centering
\begin{tabular}{c}
    \begin{tikzpicture}[scale=1]
        \node[fill=black,circle,inner sep=1pt] (t1) at (0,1) {};
        \node[fill=black,circle,inner sep=1pt] (t2) at (1,1) {};
        \node[fill=black,circle,inner sep=1pt] (t3) at (2,1) {};
        \node[fill=black,circle,inner sep=1pt] (t4) at (3,1) {};
        \node[fill=black,circle,inner sep=1pt] (t5) at (4,1) {};
        \node[fill=black,circle,inner sep=1pt] (t6) at (5,1) {};
        \node[fill=black,circle,inner sep=1pt] (t7) at (6,1) {};
        \node[fill=black,circle,inner sep=1pt] (t8) at (7,1) {};
        
        \node[fill=black,circle,inner sep=1pt] (t26) at (5,0) {};
        \node[fill=black,circle,inner sep=1pt] (t27) at (6,0) {};
        \node[fill=black,circle,inner sep=1pt] (t28) at (7,0) {};
        
        \node[fill=black,circle,inner sep=1pt] (t36) at (5,2) {};
        \node[fill=black,circle,inner sep=1pt] (t37) at (6,2) {};
        \node[fill=black,circle,inner sep=1pt] (t38) at (7,2) {};
        
        \node[fill=black,circle,inner sep=1pt] (t14) at (3,0) {};
        \node[fill=black,circle,inner sep=1pt] (t15) at (2.5,-.5) {};
        \node[fill=black,circle,inner sep=1pt] (t17) at (3.5,-.5) {};
        
        \draw (t1)--(t2);
        \draw (t3)--(t6);
        \draw (t8)--(t7);
        \draw (t4)--(t14);
        \draw [dashed] (t2)--(t3);
        \draw [dashed] (t7)--(t6);
        
        \draw [dashed] (t27)--(t26);
        \draw (t28)--(t27);
        \draw (t26)--(t5);
        
        \draw [dashed] (t37)--(t36);
        \draw (t38)--(t37);
        \draw (t36)--(t5);
        
        \draw (t15)--(t14);
        \draw (t17)--(t14);

        \node at (3,-.4) {$\ldots$};
        \node at (2,1.2) {$v$};
        \node at (4,1.2) {$u$};
        \node at (3.2,0.1) {$w$};
        \node at (6,.6) {$\underbrace{\hspace{5.6 em}}_{11 \hbox{ vertices}}$};
        \node at (1,.6) {$\underbrace{\hspace{5.6 em}}_{14 \hbox{ vertices}}$};
        \node at (3,-.9) {$\underbrace{\hspace{3 em}}_{15 \hbox{ vertices}}$};
        \node at (6,1.6) {$\underbrace{\hspace{5.6 em}}_{11 \hbox{ vertices}}$};
        \node at (6,-.4) {$\underbrace{\hspace{5.6 em}}_{11 \hbox{ vertices}}$};
        
        \node at (-.7,0) {};
        \node at (8,0) {};
        
        \end{tikzpicture}
        \end{tabular}
\caption{A tree with  $v \in C(T)$, $u \in CT(T)$, $w \in Core(T)$ which do not lie on a common path. Further, all of these middle parts are singletons.}\label{fig:ex1}
\end{figure}
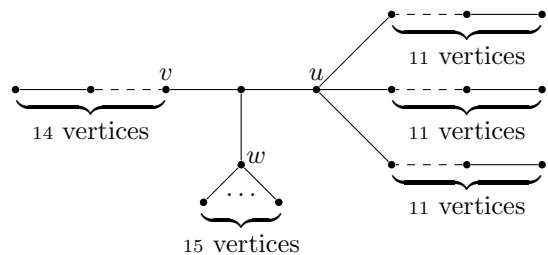 

On the other hand, when the vertices of $C(T)$, $CT(T)$, $Core(T)$ happen to lie on the same path, they can appear in any order. Figure \ref{fig:ex4} provides some illustrations. 

Among the examples with different ordering of middle vertices, it is interesting to observe that vertices in $Core(T)$  often have large degree; vertices in $C(T)$ often have  small degree; while vertices in $CT(T)$ behave somewhat between the previous two.

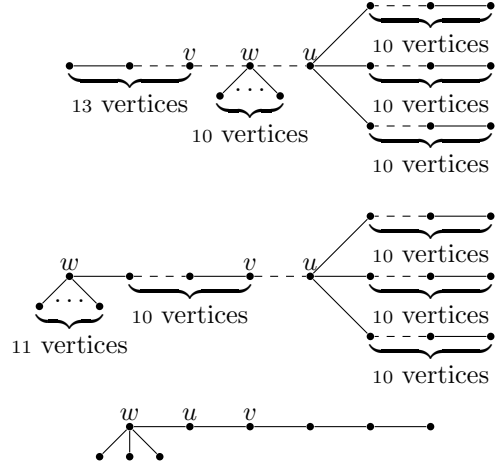
\begin{figure}[htbp]
\centering
\begin{tabular}{c}
    \begin{tikzpicture}[scale=.8]
        \node[fill=black,circle,inner sep=1pt] (t1) at (0,1) {};
        \node[fill=black,circle,inner sep=1pt] (t2) at (1,1) {};
        \node[fill=black,circle,inner sep=1pt] (t3) at (2,1) {};
        \node[fill=black,circle,inner sep=1pt] (t4) at (3,1) {};
        \node[fill=black,circle,inner sep=1pt] (t5) at (4,1) {};
        \node[fill=black,circle,inner sep=1pt] (t6) at (5,1) {};
        \node[fill=black,circle,inner sep=1pt] (t7) at (6,1) {};
        \node[fill=black,circle,inner sep=1pt] (t8) at (7,1) {};
        
        \node[fill=black,circle,inner sep=1pt] (t26) at (5,0) {};
        \node[fill=black,circle,inner sep=1pt] (t27) at (6,0) {};
        \node[fill=black,circle,inner sep=1pt] (t28) at (7,0) {};
        
        \node[fill=black,circle,inner sep=1pt] (t36) at (5,2) {};
        \node[fill=black,circle,inner sep=1pt] (t37) at (6,2) {};
        \node[fill=black,circle,inner sep=1pt] (t38) at (7,2) {};
        
        \node[fill=black,circle,inner sep=1pt] (t15) at (2.5,.5) {};
        \node[fill=black,circle,inner sep=1pt] (t17) at (3.5,.5) {};
        
        \draw (t1)--(t2);
        \draw (t5)--(t6);
        \draw [dashed] (t3)--(t5);
        \draw (t8)--(t7);

        \draw [dashed] (t2)--(t3);
        \draw [dashed] (t7)--(t6);
        
        \draw [dashed] (t27)--(t26);
        \draw (t28)--(t27);
        \draw (t26)--(t5);
        
        \draw [dashed] (t37)--(t36);
        \draw (t38)--(t37);
        \draw (t36)--(t5);
        
        \draw (t15)--(t4);
        \draw (t17)--(t4);

        \node at (3,.6) {$\ldots$};
        \node at (2,1.2) {$v$};
        \node at (4,1.2) {$u$};
        \node at (3,1.2) {$w$};
       
       \node at (6,.55) {$\underbrace{\hspace{4.6 em}}_{10 \hbox{ vertices}}$};
               \node at (1,.55) {$\underbrace{\hspace{4.6 em}}_{13 \hbox{ vertices}}$};
               \node at (3,.05) {$\underbrace{\hspace{2.5 em}}_{10 \hbox{ vertices}}$};
               \node at (6,1.55) {$\underbrace{\hspace{4.6 em}}_{10 \hbox{ vertices}}$};
               \node at (6,-.45) {$\underbrace{\hspace{4.6 em}}_{10 \hbox{ vertices}}$};

    \begin{scope}[shift={+(0,-3.5)}]
        \node[fill=black,circle,inner sep=1pt] (t1) at (0,1) {};
        \node[fill=black,circle,inner sep=1pt] (t2) at (1,1) {};
        \node[fill=black,circle,inner sep=1pt] (t3) at (2,1) {};
        \node[fill=black,circle,inner sep=1pt] (t4) at (3,1) {};
        \node[fill=black,circle,inner sep=1pt] (t5) at (4,1) {};
        \node[fill=black,circle,inner sep=1pt] (t6) at (5,1) {};
        \node[fill=black,circle,inner sep=1pt] (t7) at (6,1) {};
        \node[fill=black,circle,inner sep=1pt] (t8) at (7,1) {};
        
        \node[fill=black,circle,inner sep=1pt] (t26) at (5,0) {};
        \node[fill=black,circle,inner sep=1pt] (t27) at (6,0) {};
        \node[fill=black,circle,inner sep=1pt] (t28) at (7,0) {};
        
        \node[fill=black,circle,inner sep=1pt] (t36) at (5,2) {};
        \node[fill=black,circle,inner sep=1pt] (t37) at (6,2) {};
        \node[fill=black,circle,inner sep=1pt] (t38) at (7,2) {};
        
        \node[fill=black,circle,inner sep=1pt] (t15) at (-.5,.5) {};
        \node[fill=black,circle,inner sep=1pt] (t17) at (.5,.5) {};
        
        \draw (t1)--(t2);
        \draw (t5)--(t6);
        \draw [dashed] (t4)--(t5);
        \draw (t8)--(t7);
        \draw (t3)--(t4);

        \draw [dashed] (t2)--(t3);
        \draw [dashed] (t7)--(t6);
        
        \draw [dashed] (t27)--(t26);
        \draw (t28)--(t27);
        \draw (t26)--(t5);
        
        \draw [dashed] (t37)--(t36);
        \draw (t38)--(t37);
        \draw (t36)--(t5);
        
        \draw (t15)--(t1);
        \draw (t17)--(t1);

        \node at (0,.6) {$\ldots$};
        \node at (0,1.2) {$w$};
        \node at (4,1.2) {$u$};
        \node at (3,1.2) {$v$};
       
        \node at (6,.55) {$\underbrace{\hspace{4.6 em}}_{10 \hbox{ vertices}}$};
                      \node at (2,.55) {$\underbrace{\hspace{4.6 em}}_{10 \hbox{ vertices}}$};
                      \node at (0,.05) {$\underbrace{\hspace{2.5 em}}_{11 \hbox{ vertices}}$};
                      \node at (6,1.55) {$\underbrace{\hspace{4.6 em}}_{10 \hbox{ vertices}}$};
                      \node at (6,-.45) {$\underbrace{\hspace{4.6 em}}_{10 \hbox{ vertices}}$};

        \end{scope}

\begin{scope}[shift={+(0,-6)}]
        \node[fill=black,circle,inner sep=1pt] (t2) at (1,1) {};
        \node[fill=black,circle,inner sep=1pt] (t3) at (2,1) {};
        \node[fill=black,circle,inner sep=1pt] (t4) at (3,1) {};
        \node[fill=black,circle,inner sep=1pt] (t5) at (4,1) {};
        \node[fill=black,circle,inner sep=1pt] (t6) at (5,1) {};
        \node[fill=black,circle,inner sep=1pt] (t7) at (6,1) {};

        \node[fill=black,circle,inner sep=1pt] (t15) at (.5,.5) {};
        \node[fill=black,circle,inner sep=1pt] (t16) at (1,.5) {};
        \node[fill=black,circle,inner sep=1pt] (t17) at (1.5,.5) {};
        
        \draw (t5)--(t6);
        \draw (t4)--(t5);
        \draw (t3)--(t4);

        \draw (t2)--(t3);
        \draw (t7)--(t6);

        \draw (t15)--(t2);
        \draw (t16)--(t2);
        \draw (t17)--(t2);

        \node at (1,1.2) {$w$};
        \node at (3,1.2) {$v$};
        \node at (2,1.2) {$u$};
       
       \node at (-1.5,0) {};
       \node at (8.5,0) {};
        \end{scope}
        \end{tikzpicture}
        \end{tabular}
\caption[Trees with vertices from the center, centroid, and subtree core on a common path, but in different orders. Here all these middle parts are singletons.]{Trees with vertices $v \in C(T)$, $u \in CT(T)$, $w \in Core(T)$ on a common path, but in different orders. Here each middle part is a singleton.}\label{fig:ex4}
\end{figure}

Before proceeding to examining the largest distances between different middle parts, we formalize some necessary and sufficient conditions for a vertex to be in a middle part. Although not all of them are formally stated in the literature, we leave their relatively straightforward proofs to the reader.

\begin{prop}\label{prop:mid_c} Let $T$ be a tree with at least two vertices.
A vertex $v$ is in the center $C(T)$ if and only if there are two leaves, $u$ and $w$, such that $P(v,u)\cap P(v,w) =\{v\}$, $d(v,u)=ecc_T(v)$, and $d(v,w)\geq ecc_T(v)-1$. 
\end{prop}

\begin{cor}
If there are two leaves $u,w$ such that $d(v,w) =d(v,u)= ecc_T(v)$, then $C(T)=\{v\}$. If no such $w$ exists, then $|C(T)|=2$ where the neighbor of $v$ on $P(u,v)$ is also in the center. 
\end{cor}

Next we give a characterization of the vertices in $CT(T)$. Note that if $uv$ is an edge in tree $T$, then $T-uv$ will denote the forest that results after the deletion of edge $uv$ from $T$. 
\begin{prop}\label{prop:mid_ct} 
Let $T$ be a tree with at least two vertices.
A vertex $u$ is in the centroid $CT(T)$ if and only if for each neighbor $v$ of $u$, we have
$$ n_{uv}(v) \leq n_{uv}(u) $$
where $n_{uv}(u)$ $(n_{uv}(v))$ denotes the number of vertices in the component containing $u\;(v)$ in $T - uv$. Furthermore, if $u\in CT(T)$ and equality holds above, then $v\in CT(T)$ as well.
\end{prop}

Lastly, Proposition~\ref{prop:mid_core} gives a characterization of $Core(T)$.
\begin{prop}\label{prop:mid_core}
A vertex $u$ is in $Core(T)$ if and only if for each neighbor $v$ of $u$, we have
$$ F_{T-uv}(u) \geq F_{T-uv}(v).$$ Furthermore, if $u\in Core(T)$ and equality holds above, then $v\in Core(T)$ as well.
\end{prop}

For completeness, we provide a proof for the following simple fact.

\begin{claim}\label{claim:min_max_Core}
Among rooted trees of order $n$, the number of subtrees containing the root is at most $2^{n-1}$, achieved only by the star rooted at the center; and at least $n$, achieved only by the path rooted at one end vertex.
\end{claim}

\begin{proof}
We proceed by induction on $n$. Let $T$ be a tree of order $n$ with root $r$ and let $r$ be of degree $k$ with neighbors $v_1, \ldots, v_k$. Denote by $T_i$ the connected component containing $v_i$ in $T-\{r\}$ and let $n_i = |V(T_i)|$.

Then 
$$ F_T(r) = \prod_{i=1}^k \left( 1+ F_{T_i}(v_i) \right) \leq \prod_{i=1}^k \left( 1+ 2^{n_i-1} \right) \leq \prod_{i=1}^k 2^{n_i} = 2^{n-1}, $$
where the first inequality follows from induction hypothesis and equality holds in the second inequality if and only if $n_i=1$ for all $i$ (and consequnely $k=n-1$);
and
$$ F_T(r) = \prod_{i=1}^k \left( 1+ F_{T_i}(v_i) \right) \geq \prod_{i=1}^k \left( 1+ n_i \right) \geq 1+\sum_{i=1}^k n_i = n , $$
where the first inequality follows from induction hypothesis and equality holds in the second inequality if and only if $k=1$ and $T_1$ is a path with $v_1$ as a leaf.
\end{proof}

\section{Maximum distances between middle parts in general trees}
\label{dist middle parts}
In a tree $T$, the distance $d(S,S')$ between vertex sets $S$ and $S'$ is defined as the Hausdorff distance $\min\{d(u,v): u\in S, v\in S'\}$.
Fix an arbitrary $n\in \mathbb{Z}^+$. 
Among all trees with $n$ vertices, we determine the maximum distance that can be realized between the center, centroid, and subtree core. We will also see that these maximum distances are achieved precisely when $T$ has a ``comet'' structure.

\begin{defn}[Barefoot, Entringer, Sz\'ekely \cite{barefoot}]
An {\em $r$-comet of order $n$} is formed by attaching $n-r$ pendant vertices to one end vertex of a path on $r$ vertices (Figure~\ref{fig:r_com}).
\end{defn}

\begin{figure}
\centering
    \begin{tikzpicture}[scale=.8]
        \node[fill=black,circle,inner sep=1pt] (t2) at (1,1) {};
        \node[fill=black,circle,inner sep=1pt] (t3) at (2,1) {};
        \node[fill=black,circle,inner sep=1pt] (t4) at (3,1) {};
        \node[fill=black,circle,inner sep=1pt] (t5) at (4,1) {};
        \node[fill=black,circle,inner sep=1pt] (t6) at (5,1) {};
        \node[fill=black,circle,inner sep=1pt] (t21) at (0,1.5) {};
        \node[fill=black,circle,inner sep=1pt] (t22) at (0,.5) {};
        \node[fill=black,circle,inner sep=1pt] (t23) at (0,1.8) {};
        \node[fill=black,circle,inner sep=1pt] (t24) at (0,.2) {};
        \node[fill=black,circle,inner sep=1pt] (t25) at (0,1) {};

        \draw (t5)--(t6);
        \draw [dashed] (t3)--(t5);
        \draw (t2)--(t3);
        \draw (t21)--(t2)--(t22)--(t2)--(t23)--(t2)--(t24)--(t2)--(t25);
        \node at (3,.5) {$\underbrace{ \hspace{ 9 em } }_{\hbox{$r$ vertices }}$};
        
        
        \end{tikzpicture}
\caption{An $r$-comet of order $n$.}\label{fig:r_com}
\end{figure}
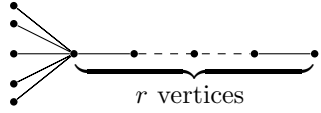

\subsection{Between center and centroid}
\label{sec:c_ct}
\begin{thm} Fix an arbitrary $n\in \mathbb{Z}^+$. 
For any tree $T$ with $n\geq 3$ vertices, 
\begin{align}
d( C(T), CT(T) ) \leq 
\left\lfloor \frac{n-3}{4} \right\rfloor.
\label{bound: c_ct}
\end{align}
\label{theo:c_ct}
\end{thm}

\begin{proof}
Fix a tree $T$ on $n$ vertices. Let $v \in C(T)$ and $u \in CT(T)$ such that the graph distance between $u$ and $v$ is precisely $d( C(T), CT(T))$. We assume $d(u,v)\geq 1$, otherwise we have nothing to prove. By the choice of $u$ and $v$, no vertex on the path $P(u,v)$ other than $u$ and $v$ is in the center or the centroid of $T$. 

Let $P(u,v)$ denote the path connecting $u$ and $v$ and let $T_u$ denote the component containing $u$ in $T-E(P(u,v))$. By Proposition~\ref{prop:mid_ct},
\[ |V(T_u)| > n - |V(T_u)| .\]
This implies
\[ |V(T_u)| > \frac{n}{2} .\] 

Let $w$ be a leaf such that $P(v,w)$ and $P(u,v)$ are disjoint, except for $v$, and the length of $P(v,w)$ is maximum. Because $v\in C(T)$ and the neighbor of $v$ on $P(u,v)$ is not in $C(T)$, Proposition \ref{prop:mid_c} tells
$$ d(v,w) = ecc_T(v). $$
As $n\geq 3$, $u$ is not a leaf. Hence it is easy to see that
$$ d(u,v) \leq ecc_T(v) - 1.  $$

Therefore, we have
\begin{eqnarray*}
\frac{n}{2} 
 >  n - |V(T_u)| \geq d(u,v) + d(v,w) \geq 2d(u,v) + 1 .
\end{eqnarray*}
This implies \[d(u,v) < \frac{ n -2 }{4 }.\]
In particular, if $n=4k+r$ with $r\in \{0,1,2\}$, then 
\[ k -\frac{1}{2} \leq \frac{ 4k+r -2 }{4} \leq  k .\]
Since $d(u,v) < \frac{ n -2 }{4 }$, when $n\equiv r \mod 4$ for $r\in \{0,1,2\}$, $d(u,v) \leq k-1$ where $k=\left\lfloor \frac{n}{4} \right\rfloor$. 

When $n=4k+3$, 
\[d(u,v) < \frac{ n -2 }{4 } = \frac{ 4k+1}{4} = k+\frac{1}{4}.\]
As a result, $d(u,v) \leq k$.
\end{proof}

\begin{prop}
Let $k:= \left\lfloor \frac{n}{4} \right\rfloor$. Equality holds in \eqref{bound: c_ct} exactly when $n$ and $T$ fall into one of the following categories: 
\begin{itemize}
\item $n = 4k$ and $T$ is the $2k$-comet.
\item $n=4k+1$  or $n=4k+2$ and $T$ is one of the following trees: 
	\begin{itemize}
	\item $2k$-comet
	\item $2k$-comet on $n-1$ vertices together with one vertex pendant to one of the degree 2 vertices on the path of the comet
	\item a tree consisting of a path on $2k+1$ vertices with an end vertex identified with the root of a height 2 tree 
	\end{itemize}
\item $n=4k+3$ and $T$ is a $(2k+2)$-comet.
\end{itemize}
\end{prop}

\begin{proof}
As in the proof of Theorem~\ref{theo:c_ct}, let $v \in C(T)$ and $u \in CT(T)$ such that the graph distance between $u$ and $v$ is precisely $d( C(T), CT(T))$. Above, we obtained $d(u,v)\leq \left\lfloor \frac{n-3}{4}\right\rfloor$ from 
\begin{align}
\left\lfloor\frac{n-1}{2}\right\rfloor & \geq  n - |V(T_u)| \label{C_CT_ineq1} \\
& \geq d(u,v) + d(v,w) \label{C_CT_ineq2} \\
& \geq 2d(u,v) + 1 . \label{C_CT_ineq3}
\end{align}

If $\left\lfloor\frac{n-1}{2}\right\rfloor$ is odd, then $d(u,v) = \left\lfloor \frac{n-3}{4}\right\rfloor$ when all of the inequalities above are tight. Equality in \eqref{C_CT_ineq2} implies that all vertices not in $T_u$ form a path $P(u,w)$. Equality in \eqref{C_CT_ineq3} implies that the height of $T_u$ is 1 because the neighbor of $v$ on $P(u,v)$ is not in $C(T)$. This, together with equality in \eqref{C_CT_ineq1} characterizes the $\left\lfloor\frac{n+1}{2}\right\rfloor$-comet. In particular, when $n=4k$, this is the $2k$-comet, and when $n=4k+3$, this is the $2k+2$-comet. 

On the other hand, notice that $2d(u,v) + 1$ is odd. So when $\left\lfloor\frac{n-1}{2}\right\rfloor$ is even, all inequalities cannot be equalities. In particular, exactly one will be strict. If \eqref{C_CT_ineq1} is the one which is strict, then we have the $\left\lfloor\frac{n-1}{2}\right\rfloor$-comet. If \eqref{C_CT_ineq2} is the one which is strict, then we have a $\left\lfloor\frac{n+1}{2}\right\rfloor$-comet on $n-1$ vertices with one extra vertex pendant to one of the degree 2 vertices. If \eqref{C_CT_ineq3} is the one which is not strict, then $T_u$ has height 2 but all $\left\lfloor\frac{n-1}{2}\right\rfloor$ vertices not in $T_u$ still lie on the path $P(u,w)$.
\end{proof}

\subsection{Between centroid and subtree core}

Next we turn our attention to the centroid and the subtree core. 
\label{sec:ct_core}

\begin{thm}
Let $T$ be a tree with $n>8$ vertices. 
If 
$n\geq 2^{\left\lceil  \log_2 n \right\rceil -1}+\left\lceil  \log_2 n \right\rceil$, 
then \[d(CT(T), Core(T)) \leq \left\lfloor  \frac{n-1}{2} \right\rfloor - \left\lfloor  \log_2 n \right\rfloor -1 \]
with equality holding if and only if $T$ is the $(n-  \left \lfloor  \log_2 n \right \rfloor -1) $-comet.
Otherwise 
\[d(CT(T), Core(T)) \leq \left\lfloor  \frac{n-1}{2} \right\rfloor - \left\lfloor  \log_2 n \right\rfloor .\]
with equality precisely when $T$ is the $ (n-  \left \lfloor  \log_2 n \right \rfloor )$-comet.
\label{theo:ct_core1} 
\end{thm}

\begin{proof}
Fix a tree $T$ with $n>8$ vertices. Let $u \in CT(T)$ and $v \in Core(T)$ where the graph distance between $u$ and $v$ is precisely $d(CT(T), Core(T))$. We assume $d(u,v)\geq 1$, otherwise we have nothing to prove.
Let $P(u,v)$ denote the path connecting $u$ and $v$ and let $T_u$, $T_v$ denote the components containing $u$, $v$ respectively in $T-E(P(u,v))$. Let $T-T_v$ be the component containing $v$ when the edges of $T_v$ are deleted from $T$. Set $x:=|V(T_u)| $ and $y:=|V(T_v)|$. 
First observe
 \[d(u,v) \leq n - x - y + 1.\]

Since $u\in CT(T)$ and the neighbor of $u$ on $P(u,v)$ is not in $CT(T)$, Proposition~\ref{prop:mid_ct} implies 
$ x > n-x $, thus $x>\frac{n}{2}$. Since $x$ is an integer, $x\geq \frac{n+1}{2}$ and consequently 
$$ x \geq \left \lceil \frac{n+1}{2} \right \rceil . $$

Similarly, since $v\in Core(T)$ and the neighbor of $v$ on $P(u,v)$ is not in $Core(T)$, Proposition~\ref{prop:mid_core} gives
$$ F_{T_v} (v) > F_{T-T_v}(w) $$
where $w$ is the unique neighbor of $v$ on $P(u,v)$. See Figure~\ref{fig:ct_core1} for an illustration of how these pieces interact.

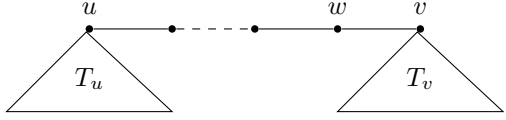
\begin{figure}[htbp]
\centering
\begin{tabular}{c}
    \begin{tikzpicture}[scale=1.1]
        \node[fill=black,circle,inner sep=1pt] (t2) at (1,1) {};
        \node[fill=black,circle,inner sep=1pt] (t3) at (2,1) {};
        \node[fill=black,circle,inner sep=1pt] (t4) at (3,1) {};
        \node[fill=black,circle,inner sep=1pt] (t5) at (4,1) {};
        \node[fill=black,circle,inner sep=1pt] (t6) at (5,1) {};

        \draw (t5)--(t6);
        \draw (t4)--(t5);
        \draw [dashed] (t3)--(t4);
        \draw (t2)--(t3);
        
        \draw (t2)--(0,0)--(2,0)--cycle;
        \draw (t6)--(4,0)--(6,0)--cycle;

        \node at (5,1.25) {$v$};
        \node at (4,1.25) {$w$};
        \node at (1,1.25) {$u$};
        
        \node at (1,.4) {$T_u$};
        \node at (5,.4) {$T_v$};
       
       \node at (-1.5,0) {};
       \node at (7.5,1) {};
        
        \end{tikzpicture}
        \end{tabular}
\caption{A representation of tree $T$ for the proof of Theorem~\ref{theo:ct_core1} with path $P(u,v)$, $T_u$, $T_v$, and $w$ labeled.}\label{fig:ct_core1}
\end{figure}

Further note that every subtree in $T_v$ which contains $v$ can be uniquely identified by the set of its vertices, excluding $v$. Thus, by Claim~\ref{claim:min_max_Core},
\[ F_{T_v}(v) \leq 2^{y-1}. \]
Note that equality holds if and only if every subset of vertices induces a tree which is the case exactly when $T_v$ is a star centered at $v$. On the other hand, 
\[ F_{T-T_v}(w) \geq n-y \]
with equality if and only if $T-T_v$ is a path with $w$ as an end vertex. 

Putting these inequalities together for our specific choice of $v$, Proposition~\ref{prop:mid_core} yields 
\[2^{y-1} > n-y.\]
As $y+2^{y-1}$ is an increasing function, there is a unique, real-valued, $y_0>0$ such that $2^{y_0-1}=n-y_0$. Further, for all $0<y<y_0$, $2^{y-1}< n-y$, and for all $y>y_0$, $2^{y-1}> n-y.$

Now let $y_0>0$, be the unique real value such that $2^{y_0-1} = n-y_0$. Then
\begin{eqnarray*}
y_0 & =& \log_2(n-y_0) + 1\\
 & < & \log_2 (n) +1.
\end{eqnarray*}
Substituting the equation $y_0 = \log_2(n-y_0) + 1$ into itself, for $n>8$ we find
\begin{eqnarray*}
y_0&=&  \log_2(n- \log_2(n-y_0) - 1) + 1 \\
 &\geq & \log_2(n- \log_2(n) - 1) + 1\\
 &= & \log_2 2(n- \log_2(n) - 1)\\
 &=& \log_2 (n+ (n-2-2\log_2 n))\\
 &>& \log_2(n).
\end{eqnarray*}
In the last inequality, $n-2-2\log_2 n>0$ holds for $n>8$. 

As a result, we have the bounds
\[ \log_2(n) < y_0 < \log_2(n)+1.\]
Further, 
if  $y_0 < \left\lfloor \log_2 n \right\rfloor +1$, then for an integer $y>0$, $2^y-1 > n-y$ precisely when $y\geq \left\lfloor \log_2 n \right\rfloor +1$. 
However, if $y_0 \geq \left\lfloor \log_2 n \right\rfloor +1$ then for an integer $y>0$, $2^y-1 > n-y$ precisely when $y\geq \left\lfloor \log_2 n \right\rfloor+2$.

When $n>8$, our bounds for integers $x$ and $y$ give
\begin{align*}
d(u,v) 
&\leq n-x- y + 1 \\
&\leq  n- \left\lceil \frac{n+1}{2} \right \rceil - \left\lfloor  \log_2 n  \right\rfloor \\
&= \left \lfloor \frac{n-1}{2} \right \rfloor - \left\lfloor  \log_2 n  \right\rfloor .
\end{align*} 
As mentioned earlier, this can be strengthen to $d(u,v) \leq  \left \lfloor \frac{n-1}{2} \right \rfloor - \left\lfloor  \log_2 n  \right\rfloor-1$ if $y_0 \geq \left\lfloor \log_2 n \right\rfloor + 1$. However, this will only happen if $2^{\left\lfloor \log_2 n \right\rfloor} \leq n-\left\lfloor \log_2 n \right\rfloor -1$ as stated in the theorem.

As for extremal trees, equality will hold in the upper bound for $d(u,v)$ exactly when $T_u$ has  $\left \lceil \frac{n+1}{2} \right \rceil$ vertices, $T_v$ is a star, and $T-T_v$ is a path. This describes the $C$-comet where $C= n-  \left \lfloor  \log_2 n \right \rfloor$ or in the case where $n\geq 2^{\left\lfloor  \log_2 n \right\rfloor}+\left\lceil  \log_2 n \right\rceil+1$, $C=n-  \left \lfloor  \log_2 n \right \rfloor -1$.
\end{proof}

\subsection{Between subtree core and center}

The study of this case is similar to that in the previous section. We omit some details.

\begin{thm}\label{theo:core_ct1}
For any tree $T$ on $n>8$ vertices, 
if $n\geq 2^{\left\lceil  \log_2 n \right\rceil -1}+\left\lceil  \log_2 n \right\rceil$ then 
\[ d(C(T), Core(T)) \leq \left\lfloor\frac{1}{2}(n-\left\lfloor  \log_2 n \right\rfloor-2) \right\rfloor\]
which is tight for the $K$-comet with $K=n-\left\lfloor  \log_2 n \right\rfloor +1$.
Otherwise
\[ d(C(T), Core(T)) \leq \left\lfloor\frac{1}{2}(n-\left\lfloor  \log_2 n \right\rfloor-1) \right\rfloor\]
which is tight for the $K$-comet with $K= n-\left\lfloor  \log_2 n \right\rfloor  $.
\end{thm}

\begin{proof}
Let $u \in Core(T)$ and $v \in C(T)$ in a tree $T$ with $|V(T)|=n$ and the graph distance between $u$ and $v$ is precisely $d(C(T), Core(T))$. 
Use $T_u$ (respectively $T_v$) to denote the component containing $u$ ($v$)  in $T-E(P(u,v))$ and let $y=|V(T_u)|$.
 
Because $v\in C(T)$ and the neighbor of $v$ on $P(u,v)$ is not in $C(T)$, there is a leaf $w$ in $T_v$ with $d(v,w) = ecc_T(v)$. 
As argued in the proof of Theorem~\ref{theo:c_ct},
$$ d(u,v) \leq ecc_T(v) - 1 < d(v,w),$$ 
$$ 2 d(u,v)+1 \leq d(u,v) + d(v,w) \leq n-y.$$
Note that these inequalities are tight for the $(n-y+1)$-comet. 

Because $u\in Core(T)$, we can conclude, as in the proof of Theorem~\ref{theo:ct_core1}, 
$$ 2^{y-1} > n - y.$$
Consequently, 
$$y\geq  \left\lfloor  \log_2 n \right\rfloor.$$ 
Combining inequalities, we obtain the bound in the theorem statement:
\[d(u,v) \leq \left\lfloor \frac{1}{2}(n-y-1)\right\rfloor \leq \left\lfloor\frac{1}{2}(n-\left\lfloor  \log_2 n \right\rfloor-1) \right\rfloor.\]

Recall from Theorem~\ref{theo:ct_core1} that if 
$n\geq 2^{\left\lceil  \log_2 n \right\rceil -1}+\left\lceil  \log_2 n \right\rceil$,
then  
$$y\geq  \left\lfloor  \log_2 n \right\rfloor +1$$
and consequently we obtain the slightly better bound
\[d(u,v) \leq \left\lfloor \frac{1}{2}(n-y-1)\right\rfloor \leq \left\lfloor\frac{1}{2}(n-\left\lfloor  \log_2 n \right\rfloor-2) \right\rfloor.\]
\end{proof}

\section{Trees with degree restrictions}
\label{deg tree}

In Section~\ref{dist middle parts}, we saw that, for each pair of middle parts, the maximum distance was achieved precisely by an appropriate comet. However, the comet has a vertex of large degree. In this section, we restrict the maximum degree of the tree and ask how the extremal structures change. We will begin with a discussion of binary trees and then broaden our scope to trees with maximum degree $k>3$. 

First we state some results about the maximum or minimum number of root-containing subtrees in a tree with a specified degree sequence which will be needed later.
Note that among trees (with no maximum degree condition) with $n$ vertices, the number of root-containing subtrees is minimized by the path, rooted at one end, and maximized by the star, rooted at the center vertex (Claim~\ref{claim:min_max_Core}). 

\subsection{Trees with a given degree sequence}

For a rooted tree, the {\it height} of a vertex is the distance to the root. The {\it height of the tree}, $h(T)$, is the maximum of all vertex heights. 

\begin{defn}\label{def:leveldegree} In a rooted tree $T$, the list of multisets $(L_0, L_1, \ldots, L_{h(T)})$, where $L_i$ consists of the degrees of the vertices at height $i$ (and $L_0$ consists of the degree of the root vertex), is called the {\it level-degree sequence} of the rooted tree.
\end{defn}

Let  $|L_i|$ be the number of entries in $L_i$ counted with multiplicity. It is easy to see that a list of multisets is the level degree sequence of a rooted tree if and only if (i) the multiset $\bigcup_i L_i$ is a degree sequence for a tree,  (ii) $|L_0|=1$, and (iii)
$\sum_{d\in L_0} d=|L_{1}|,$ while  $\sum_{d\in L_i} (d-1)=|L_{i+1}|$ for all $i\geq 1$. 

In a rooted tree, the \textit{down-degree} of the root is equal to its degree. The down degree of any other vertex is one less than its degree.

\begin{defn}\cite{nina}\label{def:greedy2} 
Given the level-degree sequence of a rooted tree, the {\it level-greedy rooted tree} for this level-degree sequence is built as follows: 
(i) For each $i\in [n]$, place $|L_i|$ vertices  in level $i$ and to each vertex, from left to right, assign a degree from $L_i$ in non-increasing order. (ii) For $i\in [n-1]$, from left to right, join the next vertex in $L_i$ whose down-degree is $d$ to the first $d$ so far unconnected vertices on level $L_{i+1}$. Repeat for $i+1$.
\end{defn}

\begin{defn}\cite{wang_dm}\label{def_greedy}
Given a tree degree sequence $(d_1, d_2, \ldots, d_n)$ in non-increasing order, the {\it greedy tree} for this degree sequence is the level-greedy tree for the level-degree sequence that has $L_0=\{d_1\}$, $L_1=\{d_2, \ldots, d_{d_1+1}\}$ and for each $i>1$, \[|L_i|=\sum_{d\in L_{i-1}} (d-1)\] with every entry in $L_{i}$ at most as large as every entry in $L_{i-1}$.
\end{defn}

The greedy tree frequently occurs in the study of extremal structures. A similar structure with modified root degree is crucial to our study here.
Fix a degree sequence for a tree and distinguish a single value in this sequence which will be the degree of the root. Similar to the greedy tree, we define the {\em rooted greedy tree}. 
\begin{defn}
Let  $\overline{d}=(d_1, d_2, \ldots, d_n)$ be a tree degree sequence in non-increasing order with degree $d_i$ identified as the root degree. Let \[\hat{d}=(d_1', d_2', \ldots, d_{n-1}')=(d_1, d_2, \ldots, \hat{d_i}, \ldots, d_n)\] be the sequence $\overline{d}$ with $d_i$ removed. The \textit{rooted greedy tree} for the degree sequence $\overline{d}$ is the level-greedy tree for the level-degree sequence that has $L_0=\{d_i\}$, $L_1=\{d_1', \ldots, d_{d_i}'\}$ and, for each $i>1$, $|L_j|=\sum_{d\in L_{j-1}} (d-1),$ where entries of $L_j$ are the next available elements from $\hat{d}$.
\end{defn}


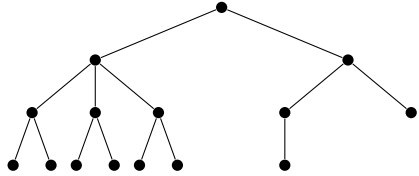
\begin{figure}[htbp]
\centering
\begin{tabular}{c}
    \begin{tikzpicture}[scale=0.7, x=1.2cm, y=.5cm]
        \node[fill=black,circle,inner sep=1.5pt] (v) at (10,6) {}; 
        \node[fill=black,circle,inner sep=1.5pt] (v1) at (4+4,4) {};

        \node[fill=black,circle,inner sep=1.5pt] (v4) at (16-4,4) {};
        \node[fill=black,circle,inner sep=1.5pt] (v11) at (3+4,2) {};
        \node[fill=black,circle,inner sep=1.5pt] (v12) at (4+4,2) {};
        \node[fill=black,circle,inner sep=1.5pt] (v13) at (5+4,2) {};
        \node[fill=black,circle,inner sep=1.5pt] (v41) at (15-4,2) {};        
        \node[fill=black,circle,inner sep=1.5pt] (v42) at (17-4,2) {};         

        \node[fill=black,circle,inner sep=1.5pt] (v111) at (2.7+4,0) {};
        \node[fill=black,circle,inner sep=1.5pt] (v112) at (3.3+4,0) {};
        \node[fill=black,circle,inner sep=1.5pt] (v121) at (3.7+4,0) {};
        \node[fill=black,circle,inner sep=1.5pt] (v122) at (4.3+4,0) {};
        \node[fill=black,circle,inner sep=1.5pt] (v131) at (4.7+4,0) {};
        \node[fill=black,circle,inner sep=1.5pt] (v132) at (5.3+4,0) {};
          
        \node[fill=black,circle,inner sep=1.5pt] (v411) at (15-4,0) {};

        \draw (v)--(v1);
        
        \draw (v)--(v4);
        \draw (v1)--(v11);
        \draw (v1)--(v12);
        \draw (v1)--(v13);
        
        \draw (v4)--(v41);
        \draw (v41)--(v411);
        \draw (v4)--(v42);
        \draw (v11)--(v111);
        \draw (v11)--(v112);
        \draw (v12)--(v121);
        \draw (v12)--(v122);
        \draw (v13)--(v131);
        \draw (v13)--(v132);




    \end{tikzpicture}
\end{tabular}
\caption[A rooted greedy tree with given degree sequence and root degree 2.]{A rooted greedy tree with root degree 2 and degree sequence $(4,3,3,3,3,2,2,1,\ldots,1)$.}
\label{rgreedy_pic}
\end{figure}

Among trees with given degree sequence, greedy trees are extremal with respect to many graph invariants. For example, the following result is for root-containing subtrees.

\begin{thm}[Andriantiana, Wagner, Wang \cite{eric}]
\label{theo:eric}
Fix a degree sequence $\overline{d}$ and a positive integer $k$. Among rooted trees with degree sequence $\overline{d}$, the greedy tree maximizes the number of subtrees with exactly $k$ vertices which contain the root. Consequently the greedy tree maximizes the total number of root-containing subtrees.

Fix a degree sequence, distinguish one value in the sequence as the root degree, and fix a positive integer $k'$. Among rooted trees with this degree sequence and the specified root degree, the corresponding rooted greedy tree maximizes the number of subtrees containing the root and which have $k'$ vertices. Consequently, the rooted greedy tree maximizes the total number of root-containing subtrees.
\end{thm}

\subsection{Binary trees}\label{sub:bin}

The study of binary trees is well motivated from its applications in phylogeny. A {\it binary tree} is a tree in which every vertex has degree 1 or 3. A {\it rooted binary tree} is a rooted tree in which the root has degree 2 and all other vertices have degree 1 or 3. 
 Sz\'ekely and Wang \cite{laszlo} studied the number of subtrees of a binary tree with labeled vertices. They found that the extremal structures are {\em good trees}, {\em rgood trees}, and {\em caterpillars}.
  In our terms, a good binary tree is a greedy tree with root degree 3 and degree sequence 
$$ (3, \ldots , 3, 1, \ldots , 1 ) $$ 
and an rgood binary tree is a rooted greedy tree with root degree 2 and degree sequence
$$ ( 3, \ldots , 3 , 2, 1, \ldots, 1 ) . $$
A binary caterpillar consists of a path $P$ with pendant vertices that make the degree of each internal vertex 3. 

Their results for the number of subtrees are as follows: 
 \begin{thm}[Sz\'ekely, Wang \cite{laszlo2}]
 Among all binary trees with $n$ leaves, the good binary tree minimizes the number of subtrees. 
 \end{thm}

\begin{thm}[Sz\'ekely, Wang \cite{laszlo}]
Among all binary trees with $n$ leaves, the binary caterpillar on $n$ leaves minimizes the number of subtrees.
\end{thm}

As an immediate consequence of Theorem~\ref{theo:eric}, we obtain the following results for root-containing subtrees.

\begin{cor}
Fix $n\in \mathbb{Z}^+$. Among all binary trees with $n$ vertices and any choice of the root, the good binary tree (with the default root of the corresponding greedy structure) has the maximum number of root-containing subtrees. 
\end{cor}

\begin{cor}\label{lem:bin1}
Fix $n\in \mathbb{Z}^{+}$. Among all rooted binary trees with $n$ vertices, the rgood binary tree is the unique tree that maximizes the number of root-containing subtrees.
\end{cor}

For binary trees, we can examine the distance between vertices of different middle parts in much the same way that we did in Section~\ref{dist middle parts}. While the exact calculations are quite messy, we believe the following is true.

\begin{conj}
Among binary trees of order $n$, the tree $T$ that maximizes $d(CT(T), C(T))$, is formed by identifying the root of an rgood binary tree with a vertex of maximum eccentricity in a binary caterpillar (Figure~\ref{fig:c_ct1_bin}). The same tree structure maximizes $d(Core(T), CT(T))$ as well as $d(Core(T), C(T))$
\label{conj_bin}
\end{conj}


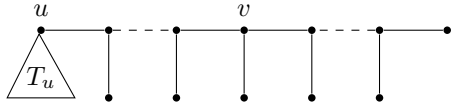
\begin{figure}[htbp]
\centering
\begin{tabular}{c}
    \begin{tikzpicture}[scale=.9]
        \node[fill=black,circle,inner sep=1pt] (t2) at (1,1) {};
        \node[fill=black,circle,inner sep=1pt] (t3) at (2,1) {};
        \node[fill=black,circle,inner sep=1pt] (t4) at (3,1) {};
        \node[fill=black,circle,inner sep=1pt] (t5) at (4,1) {};
        \node[fill=black,circle,inner sep=1pt] (t6) at (5,1) {};
        \node[fill=black,circle,inner sep=1pt] (t7) at (6,1) {};
        \node[fill=black,circle,inner sep=1pt] (t8) at (7,1) {};
        
        \node[fill=black,circle,inner sep=1pt] (t13) at (2,0) {};
        \node[fill=black,circle,inner sep=1pt] (t14) at (3,0) {};
        \node[fill=black,circle,inner sep=1pt] (t15) at (4,0) {};
        \node[fill=black,circle,inner sep=1pt] (t16) at (5,0) {};
        \node[fill=black,circle,inner sep=1pt] (t17) at (6,0) {};

        \draw (t3)--(t13);
        \draw (t4)--(t14);
        \draw (t5)--(t15);
        \draw (t6)--(t16);
        \draw (t7)--(t17);

        \draw (t5)--(t6);
        \draw (t4)--(t5);
        \draw (t8)--(t7);
        \draw [dashed] (t3)--(t4);
        \draw (t2)--(t3);
        
        \draw (t2)--(0.5,0)--(1.5,0)--cycle;

        \draw [dashed] (t7)--(t6);
        
        \node at (4,1.3) {$v$};
        \node at (1,1.3) {$u$};
        
        \node at (1,.3) {$T_u$};
       \node at (-1.5,1.3) {};
       \node at (9.5, 1.3){};
        
        \end{tikzpicture}
        \end{tabular}
\caption[An extremal binary tree structure which is conjectured to maximize the distance between each pair of middle sets.]{An extremal binary tree structure which is conjectured to maximize the distances $d(CT(T), C(T))$, $d(Core(T), CT(T))$, and $d(Core(T), C(T))$ for $u$ and $v$ as in Conjecture~\ref{conj_bin}. The tree $T_u$ is an rgood binary tree.}\label{fig:c_ct1_bin}
\end{figure} 

\subsection{Trees with bounded maximum degree}\label{sub:max}

We now turn our focus to trees on $n$ vertices, all of which have degree at most $k$. 
We previously defined good binary trees and rgood binary trees. In general, for each positive integer $k$, a {\it good tree} is a greedy tree with degree sequence \[(k,k,\ldots, k, 1,1,\ldots 1)\] while the {\it rgood trees} are  rooted greedy trees with root degree $k-1$ and degree sequence \[(k,k,\ldots, k, k-1, 1,1,\ldots 1).\] For any fixed $k$, these trees only exist for certain values of $n$. Therefore, we extend their definitions as follows so that we can create similar trees for any $n>k$.

For positive integers $n,\,k$ ($n > k$), a tree with order $n$ and maximum degree $k$ is called an {\em extended good tree} if it is a greedy tree with degree sequence \[(k,k,\ldots,k, s, 1,\ldots, 1)\] for some $1\leq s< k$ (Figure~\ref{fig:comp}). Notice that the degree sequence is determined by $n$ and $k$. By the division algorithm, we can uniquely write $n-2=q'(k-1)+s'$ with $s'<k-1$. Thus $q'$ will be the number of vertices of degree $k$, one vertex will have degree $s=s'+1$, and the rest will be leaves. 

\begin{figure}[htbp]
\centering
\begin{tabular}{c}
    \begin{tikzpicture}[scale=0.6,x=1.3cm, y=.8cm]
        \node[fill=black,circle,inner sep=1.5pt] (v) at (10,6) {}; 
        \node[fill=black,circle,inner sep=1.5pt] (v4) at (4,4) {};
        \node[fill=black,circle,inner sep=1.5pt] (v3) at (8,4) {};
        \node[fill=black,circle,inner sep=1.5pt] (v2) at (12,4) {};
        \node[fill=black,circle,inner sep=1.5pt] (v1) at (16,4) {};
        \node[fill=black,circle,inner sep=1.5pt] (v42) at (3,2) {};
        \node[fill=black,circle,inner sep=1.5pt] (v12) at (16,2) {};
        \node[fill=black,circle,inner sep=1.5pt] (v41) at (5,2) {};
        \node[fill=black,circle,inner sep=1.5pt] (v32) at (7,2) {};
        \node[fill=black,circle,inner sep=1.5pt] (v22) at (12,2) {};
        \node[fill=black,circle,inner sep=1.5pt] (v31) at (9,2) {};
        \node[fill=black,circle,inner sep=1.5pt] (v23) at (11,2) {};
        \node[fill=black,circle,inner sep=1.5pt] (v21) at (13,2) {};
        \node[fill=black,circle,inner sep=1.5pt] (v13) at (15,2) {};        
        \node[fill=black,circle,inner sep=1.5pt] (v11) at (17,2) {};         
        \node[fill=black,circle,inner sep=1.5pt] (v33) at (8,2) {};        
        \node[fill=black,circle,inner sep=1.5pt] (v43) at (4,2) {}; 
        
        \node[fill=black,circle,inner sep=1.5pt] (v111) at (2.7,0) {};
        \node[fill=black,circle,inner sep=1.5pt] (v112) at (3.3,0) {};
        \node[fill=black,circle,inner sep=1.5pt] (v121) at (3.7,0) {};
        \node[fill=black,circle,inner sep=1.5pt] (v122) at (4.3,0) {};
        \node[fill=black,circle,inner sep=1.5pt] (v131) at (4.7,0) {};
        \node[fill=black,circle,inner sep=1.5pt] (v132) at (5.3,0) {};
        \node[fill=black,circle,inner sep=1.5pt] (v211) at (6.7,0) {};
        \node[fill=black,circle,inner sep=1.5pt] (v212) at (7.3,0) {};
        \node[fill=black,circle,inner sep=1.5pt] (v221) at (7.7,0) {};
        \node[fill=black,circle,inner sep=1.5pt] (v222) at (8.3,0) {};
        \node[fill=black,circle,inner sep=1.5pt] (v231) at (9,0) {};        
        
        \node[fill=black,circle,inner sep=1.5pt] (v113) at (3,0) {};
        \node[fill=black,circle,inner sep=1.5pt] (v123) at (4,0) {};
        \node[fill=black,circle,inner sep=1.5pt] (v133) at (5,0) {};  
        
        \node[fill=black,circle,inner sep=1.5pt] (v213) at (7,0) {};
        \node[fill=black,circle,inner sep=1.5pt] (v223) at (8,0) {};

        \draw (v)--(v1);
        \draw (v)--(v2);
        \draw (v)--(v3);
        \draw (v)--(v4);
        \draw (v1)--(v11);
        \draw (v1)--(v12);
        \draw (v1)--(v13);
        \draw (v2)--(v21);
        \draw (v2)--(v22);
        \draw (v2)--(v23);
        \draw (v3)--(v31);
        \draw (v3)--(v32);
        \draw (v4)--(v41);
        \draw (v4)--(v42);
        \draw (v42)--(v111);
        \draw (v42)--(v112);
        \draw (v43)--(v121);
        \draw (v43)--(v122);
        \draw (v41)--(v131);
        \draw (v41)--(v132);
        \draw (v32)--(v211);
        \draw (v32)--(v212);
        \draw (v33)--(v221);
        \draw (v33)--(v222);
        \draw (v31)--(v231);
        
        \draw (v42)--(v113);
        \draw (v43)--(v123);
        \draw (v41)--(v133);
        \draw (v32)--(v213);
        \draw (v33)--(v223);
        
        \draw (v3)--(v33);
        \draw (v4)--(v43);

    \end{tikzpicture}
\end{tabular}
\caption{An extended good tree with 33 vertices and maximum degree 4.}
\label{fig:comp}
\end{figure}
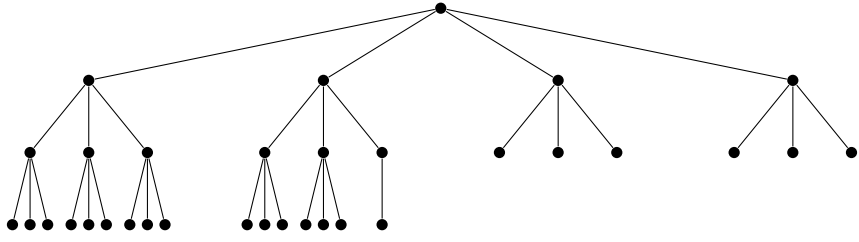

Similarly, for positive integers $n,k$, the {\em extended rgood tree} with order $n$ and maximum degree $k$, is a rooted greedy tree with root degree $k-1$ and degree sequence \[(k,k,\ldots,k, k-1, s, 1,\ldots, 1)\] for some $1\leq s< k$ (Figure~\ref{fig:rcomp}).  

\begin{figure}[htbp]
\centering
\begin{tabular}{c}
    \begin{tikzpicture}[scale=0.5, x=1.4cm,y=.8cm]
        \node[fill=black,circle,inner sep=1.5pt] (v) at (10,6) {}; 
        \node[fill=black,circle,inner sep=1.5pt] (v4) at (4,4) {};
        \node[fill=black,circle,inner sep=1.5pt] (v3) at (8+2,4) {};

        \node[fill=black,circle,inner sep=1.5pt] (v1) at (16,4) {};
        \node[fill=black,circle,inner sep=1.5pt] (v42) at (3,2) {};
        \node[fill=black,circle,inner sep=1.5pt] (v12) at (16,2) {};
        \node[fill=black,circle,inner sep=1.5pt] (v41) at (5,2) {};
        \node[fill=black,circle,inner sep=1.5pt] (v32) at (7+2,2) {};
        
        \node[fill=black,circle,inner sep=1.5pt] (v31) at (9+2,2) {};
        \node[fill=black,circle,inner sep=1.5pt] (v13) at (15,2) {};        
        \node[fill=black,circle,inner sep=1.5pt] (v11) at (17,2) {};         
        \node[fill=black,circle,inner sep=1.5pt] (v33) at (8+2,2) {};        
        \node[fill=black,circle,inner sep=1.5pt] (v43) at (4,2) {}; 
        
        \node[fill=black,circle,inner sep=1.5pt] (v111) at (2.7,0) {};
        \node[fill=black,circle,inner sep=1.5pt] (v112) at (3.3,0) {};
        \node[fill=black,circle,inner sep=1.5pt] (v121) at (3.7,0) {};
        \node[fill=black,circle,inner sep=1.5pt] (v122) at (4.3,0) {};
        \node[fill=black,circle,inner sep=1.5pt] (v131) at (4.7,0) {};
        \node[fill=black,circle,inner sep=1.5pt] (v132) at (5.3,0) {};
        \node[fill=black,circle,inner sep=1.5pt] (v211) at (6.7+2,0) {};
        \node[fill=black,circle,inner sep=1.5pt] (v212) at (7.3+2,0) {};
        \node[fill=black,circle,inner sep=1.5pt] (v221) at (7.7+2,0) {};
        \node[fill=black,circle,inner sep=1.5pt] (v222) at (8.3+2,0) {};
        \node[fill=black,circle,inner sep=1.5pt] (v231) at (9+2,0) {};        
        
        \node[fill=black,circle,inner sep=1.5pt] (v113) at (3,0) {};
        \node[fill=black,circle,inner sep=1.5pt] (v123) at (4,0) {};
        \node[fill=black,circle,inner sep=1.5pt] (v133) at (5,0) {};  
        
        \node[fill=black,circle,inner sep=1.5pt] (v213) at (7+2,0) {};
        \node[fill=black,circle,inner sep=1.5pt] (v223) at (8+2,0) {};

        \draw (v)--(v1);

        \draw (v)--(v3);
        \draw (v)--(v4);
        \draw (v1)--(v11);
        \draw (v1)--(v12);
        \draw (v1)--(v13);
        \draw (v3)--(v31);
        \draw (v3)--(v32);
        \draw (v4)--(v41);
        \draw (v4)--(v42);
        \draw (v42)--(v111);
        \draw (v42)--(v112);
        \draw (v43)--(v121);
        \draw (v43)--(v122);
        \draw (v41)--(v131);
        \draw (v41)--(v132);
        \draw (v32)--(v211);
        \draw (v32)--(v212);
        \draw (v33)--(v221);
        \draw (v33)--(v222);
        \draw (v31)--(v231);
        
        \draw (v42)--(v113);
        \draw (v43)--(v123);
        \draw (v41)--(v133);
        \draw (v32)--(v213);
        \draw (v33)--(v223);
        
        \draw (v3)--(v33);
        \draw (v4)--(v43);

    \end{tikzpicture}
\end{tabular}
\caption{An extended rgood tree with 29 vertices and maximum degree 4.}
\label{fig:rcomp}
\end{figure}
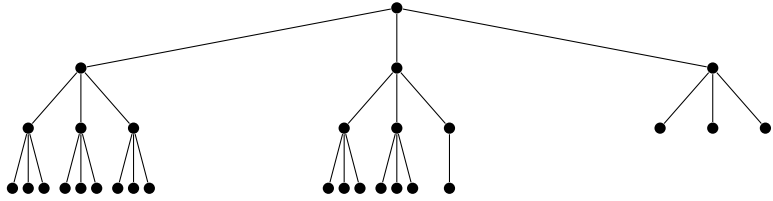

Among all rooted trees with $n$ vertices, maximum degree $k$, and root degree $\rho$, $1\leq \rho \leq k-1$, we seek the one with the maximum number of root-containing subtrees.

\begin{thm}\label{theo:kary}
Among all rooted trees with $n$ vertices, maximum degree $k$, and root degree $\rho$ where $1\leq \rho \leq k-1$, the extended rgood tree maximizes the number of root-containing subtrees.
\end{thm}

Theorem~\ref{theo:kary} follows from Lemmas~\ref{lem:kary} and \ref{lem:kary0} below.  

\begin{lemma}\label{lem:kary}
For any pair $(n,k)$ of positive integers, $n>k$, a tree with $n$ vertices, maximum degree $k$, and root degree $k-1$ which  maximizes the number of root-containing subtrees must have root degree $k-1$. 
\end{lemma}

\begin{proof}
For contradiction, suppose $T$ is such a tree with root $r$ having degree $\rho\leq k-2$. Since $n \geq k$, there exists a child $u$ of $r$ that is not a leaf. Let $v$ be a child of $u$. 

Define $T':=T-\{uv\} + \{rv\}$. Every root-containing subtree in $T$ can be uniquely identified by its list of vertices. It is easy to see that each list forms a root-containing subtree in $T'$. However, $T'$ also has root-containing subtrees which contain $v$ and not $u$. These do not appear in $T$. Therefore $T'$ has more root-containing subtrees than $T$. This contradicts our choice of $T$. 
\end{proof}

\begin{defn}
For positive integer sequences $\pi=(d_0,\cdots,
d_{n-1})$ and $\pi^{\prime}=(d_0^{\prime}, \cdots, d_{n-1}^{\prime})$, we say $\pi^{\prime}$ \textit{majorizes} $\pi$, denoted $ \pi\triangleleft \pi^{\prime}$,  if for each $k\in \{0, \cdots, n-2\}$,
 \[\sum_{i=0}^{k}d_i\le\sum_{i=0}^kd_i^{\prime} \qquad \text{and} \qquad \sum_{i=0}^{n-1}d_i=\sum_{i=0}^{n-1}d_i^{\prime}.\]
 \end{defn}

The following is a simpler analogue of Theorem 11 of Andriantiana, Wagner, and Wang \cite{eric}. 
We skip the details.

\begin{lemma}\label{lem:kary0}
Let $T$ and $T'$ be rooted greedy trees on $n$ vertices with root degree $k-1$. If $T$ has degree sequence $\pi$ and $T'$ has degree sequence $\pi'$ where 
$ \pi\triangleleft \pi^{\prime} $, then $T'$ has more root-containing subtrees than $T$.
\end{lemma}

In the search for a tree which  maximizes the number of root-containing subtrees, Lemma~\ref{lem:kary} implies that it is sufficient to restrict our attention to trees with root degree $k-1$. Because we are considering only degree sequences on $n$ vertices with maximum degree $k$, it is easy to see that the degree sequence of the extended rgood tree majorizes all other such degree sequences. Thus, Lemma~\ref{lem:kary0} then implies that the extended rgood tree for order $n$ and maximum degree $k$ as stated in Theorem~\ref{theo:kary}.

For the purpose of our study of maximum distances between different middle parts, we also note the following fact.

\begin{rem}\label{rem:k1}
Among all rooted trees of given order, root degree at most $k-1$, and maximum degree $k$:
\begin{itemize}
\item the extended rgood tree minimizes the height;
\item the path (rooted at one end) minimizes the number of root-containing subtrees and maximizes the height.
\end{itemize}
\end{rem}

\subsection{Middle parts in trees with a given maximum degree}

Fix $n,k\in \mathbb{Z}^{+}$. Similar to the binary tree case, we restrict our attention to classes of trees which have order $n$ and  maximum degree $k$. In this section, we detail our findings for the trees in this class which maximize the distance between different middle parts.

\begin{thm}
For fixed $n,k\in \mathbb{Z}^{+}$, each tree $T$ with order $n$ and maximum degree $k$ has
\[d(CT(T), C(T)) \leq \frac{n-\left\lceil\frac{n+1}{2} \right\rceil-h_u}{2}\]
where \[h_u=\left\lceil\frac{ \ln\left( \lceil\frac{n+1}{2}\rceil(k-2)+1 \right)}{\ln(k-1)}\right\rceil-1.\]
This inequality is tight for the tree formed by identifying the root of an extended rgood tree with one end of a  path of appropriate length.  
\label{max_deg_ct_c}
\end{thm}

\begin{proof}
For a fixed tree $T$, select $u\in CT(T)$ and $v\in C(T)$ such that $d(u,v) = d(CT(T), C(T))$. Assume $d(u,v)\geq 1$, otherwise there is nothing to prove. Let $T_u$ and $T_v$ name the components containing $u$ and $v$ respectively in $T - E(P(u,v))$. 

Counting the vertices in $T$, we obtain the inequality
\begin{equation}\label{eq:max1}
d(u,v)\leq n-|V(T_u)|-|V(T_v)|+1.
\end{equation}

Because $u\in CT(T)$, Proposition~\ref{prop:mid_ct} implies $ |V(T_u)| > n- |V(T_u)|$ and hence
$$ |V(T_u)| \geq \left\lceil \frac{n+1}{2} \right\rceil . $$

Set $h_u$ and $h_v$ equal to the heights of $T_u$ and $T_v$ respectively. Because $v\in C(T)$, Proposition~\ref{prop:mid_c} implies $d(u,v)+h_u \leq h_v$ and hence
\begin{equation}\label{eq:max2} 
d(u,v) \leq h_v - h_u \leq |V(T_v)|-1 -h_u.
\end{equation}
The upper bound for $d(u,v)$ is tight when $h_v = |V(T_v)|-1$, which happens exactly when $T_v$ is a path, and $h_u$ is minimum. 

By Remark~\ref{rem:k1}, the minimum $h_u$ is achieved when $T_u$ is the extended rgood tree. Since $|V(T_u)|\geq \left\lceil \frac{n+1}{2} \right\rceil$ and the maximum degree is $k$, we can determine the height of an extended rgood tree with these conditions. The extended rgood tree with maximum degree $h$ and height $h$ has at most $\sum_{i=0}^{h} (k-1)^i$ vertices. For $T_u$ with $n$ vertices, the height $h$ will be the smallest value which satisfies
\begin{align*}
|V(T_u)| & \leq \sum_{i=0}^{h} (k-1)^i \\
&= \frac{ (k-1)^{h+1}-1}{k-2} .
\end{align*}
As a result,  
\begin{align*}
h \geq \frac{\ln \left( |V(T_u)|(k-2) +1\right)}{\ln (k-1)} -1.
\end{align*}
Since $h$ is the smallest value that satisfies the above inequality and $|V(T_u)| = \left\lceil \frac{n+1}{2} \right\rceil$, we can conclude
\[h=\left\lceil\frac{ \ln\left( \lceil\frac{n+1}{2}\rceil(k-2)+1 \right)}{\ln(k-1)}\right\rceil-1.\]

Without knowing $|V(T_v)|$ exactly, we can add \eqref{eq:max1} and \eqref{eq:max2} and solve for $d(u,v)$ to obtain the desired upper bound for $d(u,v)$:
\begin{align*}
2d(u,v)  & \leq n-|V(T_u)| - h_u\\
d(u,v) & \leq \frac{1}{2}\left( n-|V(T_u)| - h_u \right)\\
    & \leq \left\lfloor \frac{1}{2}\left( n-\left\lceil\frac{n+1}{2} \right\rceil - h \right) \right\rfloor.
 \end{align*}
\end{proof}

\begin{thm}
For fixed $n,k\in \mathbb{Z}^{+}$, each tree $T$ with order $n$ and maximum degree $k$ has
\[d(Core(T), CT(T)) \leq n-n'-\left\lceil\frac{n+1}{2}\right\rceil +1\]
where $n'$ is the minimum order of an extended rgood tree $T_u$ with maximum degree $k$ such that $F_{T_u}(u) \geq n-|V(T_u)|$.
This inequality is tight for the tree formed by identifying the root of the extended rgood tree with one end of a path of appropriate length.  
\end{thm}

\begin{proof}
Let $u\in Core(T)$ and $v\in CT(T)$ such that $d(u,v) = d(Core(T), CT(T))$. Assume $d(u,v)\geq 1$, otherwise there is nothing to prove.
Define $T_u$ and $T_v$ to be the components of $T-E(P(u,v))$ containing $u$ and $v$ respectively. Let $h_u$ and $h_v$ be the heights of $T_u$ and $T_v$ respectively. 

Similar to before, Proposition~\ref{prop:mid_ct} implies
$$ |V(T_v)| \geq \left\lceil\frac{n+1}{2}\right\rceil. $$

By Proposition~\ref{prop:mid_core}, $u\in Core(T)$ and its neighbor $w$ on $P(u,v)$ is not in the subtree core precisely when 
\begin{align}
F_{T_u}(u) &\geq 1 + F_{T-T_u}(w) \geq d(u,v) + F_{T_v}(v) \geq d(u,v) + |V(T_v)| ,\nonumber \\
d(u,v) &\leq F_{T_u}(u) -F_{T_v}(v) \leq F_{T_u}(u) - |V(T_v)| . \label{eq:max3} 
\end{align}
The last inequality is tight if $T_v$ is a path. 

Counting the vertices in $T$, we see
\begin{align*}
n  &\geq d(u,v) + |V(T_u)| + |V(T_v)| -1, \\
 d(u,v) & \leq n- |V(T_u)| - |V(T_v)| +1 \leq n- \left\lceil\frac{n+1}{2}\right\rceil - n' +1. 
 \end{align*}
where $n'$ is the minimum number of vertices in a tree $T_u$ with maximum degree $k$ such that 
$F_{T_u}(u) \geq  d(u,v) + |V(T_v)| = n- |V(T_u)|$ as in \eqref{eq:max3}. Note that $F_{T_u}(u)$ is maximized by the extended rgood tree, giving the extremal tree in the theorem statement. 
\end{proof}

\begin{thm}
For fixed $n,k\in \mathbb{Z}^{+}$, each tree $T$ with order $n$ and maximum degree $k$ has
\[d(Core(T), C(T)) \leq n-n' - \left\lfloor \frac{1}{2}(n-n' + h') \right\rfloor\]
where $h' = \left\lceil \frac{\ln\left( n'(k-2)+1 \right)}{\ln(k-1)}\right\rceil-1$ and $n'$ is the minimum number of vertices in the extended rgood tree $T_u$ with maximum degree $k$  such that 
$F_{T_u}(u) \geq n- |V(T_u)| $.
This inequality is tight for the tree formed by identifying the root of the extended rgood tree with one end of a path of appropriate length.  
\end{thm}

\begin{proof}
Let $u\in Core(T)$ and $v\in C(T)$ such that $d(u,v) = d(Core(T), C(T))$. 
Define $T_u$ and $T_v$ to be the components of $T-E(P(u,v))$ containing $u$ and $v$ respectively. Let $h_u$ and $h_v$ be the heights of $T_u$ and $T_v$ respectively. 

Because $u\in Core(T)$ and its neighbor on $P(u,v)$ is not in the subtree core, as in \eqref{eq:max3}, Proposition~\ref{prop:mid_core} gives 
\begin{align}
d(u,v) \leq F_{T_u}(u) - |V(T_v)|\label{eq:max3'}
\end{align}
which is tight when $T_v$ is a path. 

Because $v\in C(T)$ and its neighbor on $P(u,v)$ is not in the center, as in the proof of Theorem~\ref{max_deg_ct_c},  Proposition~\ref{prop:mid_c} gives
\[d(u,v) \leq h_v - h_u \leq |V(T_v)| - h_u - 1.\]
As in \eqref{eq:max2}, this is also tight when $T_v$ is a path. 

Adding these two inequalities together we obtain the following bound. 
$$ d(u,v) \leq \frac12 \left( F_{T_u}(u) - h_u -1 \right) .$$
The upper bound is maximum when $F_{T_u}(u)$ large and $h_u$ is small which is optimized when $T_u$ is the extended rgood tree.

If $n'$ is the number of vertices in $T_u$, then because $v\in C(T)$ and $T_v$ is a path, 
then $ecc_T(v)$ is at least half of the diameter of $T$ which translates to
\begin{align*}
 |V(T_v)| \geq \frac{1}{2}(n-n' + h_u).
 \end{align*}
 
Any tree on $n'$ vertices with maximum degree at most $k$ will have height at least the height of the corresponding extended rgood tree. As determined in the proof of Theorem~\ref{max_deg_ct_c}, 
 \begin{align*} h_u \geq \left\lceil \frac{\ln\left( n'(k-2)+1 \right)}{\ln(k-1)}\right\rceil-1:=h' .\label{h_u} \end{align*}
 In conclusion, 
\begin{align*}
d(u,v) \leq n-|V(T_u)| - |V(T_v)| \leq n-n' - \left\lfloor \frac{1}{2}(n-n' + h') \right\rfloor.
\end{align*}
Further, this upper bound is maximized when $n'$ is minimized. However, $n'$ must still satisfying the condition $F_{T_u}(u) \geq  d(u,v) + |V(T_v)| = n- |V(T_u)|$ from \eqref{eq:max3'}.
\end{proof}

\section{Different middle parts in trees with a given diameter \texorpdfstring{$D$}{D}}
\label{diam tree}
Recall that all extremal trees in Section 2 were comets. In the previous section, we explored the effect of limiting the maximum vertex degree in a tree. Here, we ask how the distances between middle parts and the corresponding extremal structures change when the diameter is limited. 
The next two propositions follow from exactly the same arguments as those for Theorem~\ref{theo:c_ct} and Theorem~\ref{theo:core_ct1} in Section~\ref{dist middle parts}, we skip the proofs.

\begin{prop}\label{theo:c_ct2}
For fixed integers $D\geq 2$ and there exists $n_0$ such that for all $n>n_0$, every tree $T$ of order $n$ and diameter at most $D$ satisfies 
$$ d(C(T), CT(T)) \leq \left\lfloor \frac{D-2}{2}\right\rfloor ,$$
which is achieved by a $D$-comet.
\end{prop}

\begin{prop}\label{theo:c_core2}
For fixed $D\geq 2$, there exists $n_0$ such that for all $n>n_0$, every tree $T$ of order $n$ and diameter at most $D$ satisfies $$ d(C(T), Core(T)) \leq  \left\lfloor \frac{D-2}{2}\right\rfloor ,$$
which is achieved by a $D$-comet.
\end{prop}

Finding the maximum of $d(CT(T),Core(T))$ is unexpectedly more difficult. Notwithstanding the progress in the remaining part of this paper, it remains unsolved.
Fix integers $D\geq 2$ and $n\geq 1$. Among all trees with diameter at most $D$ and order $n$, fix a tree $T$ which realizes the maximum value for $d(CT(T), Core(T))$. 

Select vertices $u\in Core(T)$  and $v\in CT(T)$  such that the graph distance between $u$ and $v$ is precisely $d(CT(T), Core(T))$. We assume $d(u,v) \geq 1$, otherwise there is nothing to prove. 
In $T- E(P(u,v))$, let $T_u$ name the component containing $u$ while $T_v$ is the component containing $v$. Consider $u$ to be the root of $T_u$ and $v$ to be the root of $T_v$. 

Let $w$ be the neighbor of $u$ on $P(u,v)$. Because $u\in Core(T)$ and $w \not\in Core(T)$, Proposition~\ref{prop:mid_core} implies
\[ F_{T_u}(u) < F_{T-T_u}(w).\]

Because $v\in CT(T)$ and its neighbor on $P(u,v)$ is not in $CT(T)$, Proposition~\ref{prop:mid_ct} implies
\[ |V(T_v)| > n- |V(T_v)|.\]
 
Suppose $T_u$ is not a star. Create a new tree $T'$ from $T$ by replacing $T_u$ with a star $T'_u$ which is rooted at $u$ and has the same order as $T_u$.
Using the convention that $T'_u$ and $T'_v$ are the components containing $u$ and $v$ respectively in $T' - E(P(u,v))$, we see that $T'_v$ is isomorphic to $T_v$.
 First observe that
\[F_{T'_u}(u) \geq F_{T_u}(u) > F_{T-T_u}(w) = F_{T'-T'_u}(w)\]
which implies $w\not\in Core(T')$ by Proposition~\ref{prop:mid_core}.  
Further,  
\[|V(T'_v)| = |V(T_v)|  > n-|V(T_v)| = n-|V(T'_v)|\]
which implies the neighbor of $v$ on $P(u,v)$ is not in the centroid of $T'$ by Proposition~\ref{prop:mid_ct}. Therefore \[d(Core(T'), CT(T')) \geq d_{T'}(u,v) = d_T(u,v) = d(Core(T), CT(T)).\] By the choice of $T$, $d(Core(T'), CT(T')) = d(Core(T), CT(T))$. So $T'$ is also a tree with diameter at most $D$ and order $n$ which maximizes $d(Core(T), CT(T))$. 

Now consider the structure of $T'_v$ in $T'$. Say $T'_v$ has $x$ vertices and height $h$. Suppose $T'_v$ does not minimize the number of subtrees containing $v$ for its height and order. Let $T''_v$ be a tree rooted at $v$ with height at most $h$ and order $x$ which minimizes $F_{T''_v}(v)$. Define $T''$ to be the tree created from $T'$ by replacing $T'_v$ with $T''_v$. Observe that 
\[F_{T''_u}(u) = F_{T'_u}(u) > F_{T'-T'_u}(w) > F_{T''-T''_u}(w)\]
which implies $w\not\in Core(T'')$ by Proposition~\ref{prop:mid_core}.
Further,  for $T''_v$ being the component of $T''- E(P(u,v))$ which contains $v$,
\[|V(T''_v)| = |V(T'_v)|  > n-|V(T'_v)| = n-|V(T''_v)|.\]
This implies, by Proposition~\ref{prop:mid_ct}, that the neighbor of $v$ on $P(u,v)$ in $T''$ is not in $CT(T'')$ and \[d(Core(T''), CT(T'')) \geq d_{T''}(u,v) = d_T(u,v) = d(Core(T), CT(T)).\] By the choice of $T$, $d(Core(T''), CT(T''))  = d(Core(T), CT(T))$ which implies $T''$ is also a tree with diameter at most $D$ and order $n$ that maximizes the distance between the subtree core and the centroid. 

\begin{rem}
Fix $n,D\in \mathbb{Z}^{+}$. Among all trees with diameter at most $D$ and order $n$, one such tree $T$ which maximizes $d(Core(T), CT(T))$ has $T_u$ being a star rooted at $u$ and $T_v$ a tree which minimizes the number of subtrees containing $v$ for its height and order. This structure $T$ is drawn in Figure~\ref{diam_CT_Core}.
\end{rem}

 In Section~\ref{sec: Min_subtrees}, we take a closer look at the structure of $T_v$, a tree which minimizes the number of subtrees containing $v$ for its height and order. While we determine many necessary properties of $T_v$, characterizing the exact structure is still an open problem.

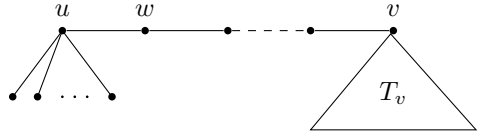
\begin{figure}[htbp]
\centering
\begin{tabular}{c}
    \begin{tikzpicture}[scale=1.1]
        \node[fill=black,circle,inner sep=1pt] (t2) at (1,1) {};
        \node[fill=black,circle,inner sep=1pt] (t3) at (2,1) {};
        \node[fill=black,circle,inner sep=1pt] (t4) at (3,1) {};
        \node[fill=black,circle,inner sep=1pt] (t5) at (4,1) {};
        \node[fill=black,circle,inner sep=1pt] (t6) at (5,1) {};

        \draw (t5)--(t6);
        \draw [dashed] (t4)--(t5);
        \draw (t3)--(t4);
        \draw (t2)--(t3);
        
        \foreach \x in {-2,-1,2} \node [fill=black,circle,inner sep=1pt] at (1+\x*.3,0.2) {};
        \foreach \x in {-2,-1,2} \draw (1+\x*.3,0.2)--(1,1);
        \foreach \x in {0,.5,1} \node [fill=black,circle,inner sep=.3pt] at (1+\x*.3,0.2) {};
        \draw (t6)--(4,-.2)--(6,-.2)--cycle;

        \node at (5,1.25) {$v$};
        \node at (2,1.25) {$w$};
        \node at (1,1.25) {$u$};
        
        \node at (5,.25) {$T_v$};
       
        \node at (-1.5,1) {};
        \node at (7.5,1) {};
        \end{tikzpicture}
        \end{tabular}
\caption[The structure of a tree $T$ with diameter $D$ and order $n$ which maximize $d(Core(T), CT(T))$.]{The structure of a tree $T$ with diameter $D$ and order $n$ which maximize $d(Core(T), CT(T))$. Here, $u\in Core(T)$, $v\in CT(T)$, and $T_v$ minimizes the number of subtrees containing $v$ for its order and height.}
\label{diam_CT_Core}
\end{figure}

\section{Rooted trees of given order and height}
\label{sec: Min_subtrees}

For any $n,h\in \mathbb{Z}^{+}$, this section is devoted to characterizing the rooted trees with $n$ vertices and height at most $h$ which have the minimum number of root-containing subtrees. For the remainder of this section, we will call these trees \emph{optimal}.

To standardize some notation, we restrict our attention to trees $T$ which are rooted at root $\rho$, have order $n$ and height at most $h$ unless mentioned otherwise. Note that $h(T)=ecc_T(\rho)$. The degree of a vertex $v$ will be denoted $deg(v)$. 

For any $v\in V(T)$, let $T(v)$ denote the subtree induced by $v$ and all of its descendants. We will view $T(v)$ as a tree rooted at $v$. For each neighbor $v_i$ of $\rho$, set $T_i:=T(v_i)$.  
For $f \in \{0,1, \ldots, h-1\}$, define the \textit{$f$-split} to be the tree rooted at $w_1$ with $h+f$ vertices, constructed 
from paths $P_1=(w_1, w_2, \ldots, w_h)$, and $P_2=(u_1, u_2, \ldots, u_f)$ by adding the edge $u_1w_{h-f}$. In other words, the midpoint of path on $2f$ edges is joined to the root $w_1$ by an $(h-f-1)$-edge path. Our main results from this section are summarized in the following theorem.

\begin{thm}
For positive integers $n$ and $h$, there is an optimal tree $T$ with $n$ vertices and height at most $h$ such that each $T_i$ is a $k_i$-split with $k_1 \geq k_2 \geq \ldots \geq k_r$ (see Figure~\ref{fig:ex1_n}). Further, the tuple $(k_1, k_2, \ldots, k_r)$ is described by one of the following three types: 
\begin{enumerate}
\item[(i)] (Paths) $k_r=0$, $k_{r-1}\in \{0,1\}$, $k_{r-2} = \ldots = k_1 =1$.
\item[(ii)] (One large) $k_1 > \left \lceil \sqrt{h+\frac{5}{4}}-\frac{1}{2}\right\rceil$ and $k_i = \left\lfloor \frac{h+1}{k_1+1} \right\rfloor$ for each $i\in \{2,\ldots, r\}$ provided $\left\lfloor \frac{h+1}{k_1+1} \right\rfloor \geq \frac{h+1}{k_1}-1$.
Further, if $n>5h^2$, then $k_i\leq \ln(6h)$ for each $i\in \{2,\ldots, r\}$.
\item[(iii)] (Even distribution) $k_1 \leq \left \lceil \sqrt{h+\frac{5}{4}}-\frac{1}{2}\right\rceil$ and for all $i,j\in [r]$, $|k_i - k_j| \leq 1$.
Further, if $n>5h^2$, then $k_i\leq \ln(6h)$ for each $i\in \{2,\ldots, r\}$.
\end{enumerate}
\label{thm:optimal_summary}
\end{thm}

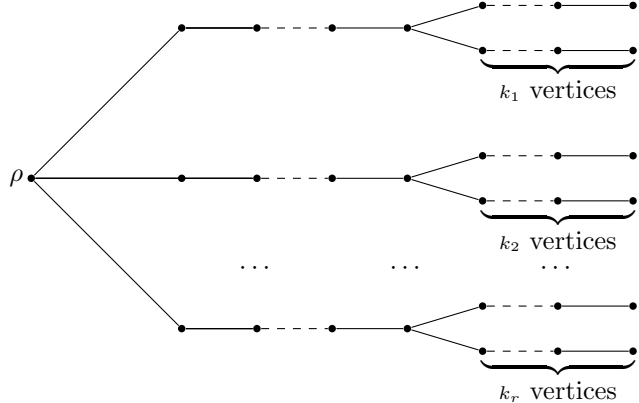
\begin{figure}[htbp]
\centering
\begin{tabular}{c}
    \begin{tikzpicture}[scale=1]
        \node[fill=black,circle,inner sep=1pt] (t1) at (0,4) {};
        \node[fill=black,circle,inner sep=1pt] (t2) at (2,4) {};
        \node[fill=black,circle,inner sep=1pt] (t3) at (3,4) {};
        \node[fill=black,circle,inner sep=1pt] (t4) at (4,4) {};
        \node[fill=black,circle,inner sep=1pt] (t5) at (5,4) {};

        \node[fill=black,circle,inner sep=1pt] (t6) at (7-1,3.7) {};
        \node[fill=black,circle,inner sep=1pt] (t7) at (8-1,3.7) {};
        \node[fill=black,circle,inner sep=1pt] (t8) at (9-1,3.7) {};

        \node[fill=black,circle,inner sep=1pt] (t9) at (7-1,4.3) {};
        \node[fill=black,circle,inner sep=1pt] (t10) at (8-1,4.3) {};
        \node[fill=black,circle,inner sep=1pt] (t101) at (9-1,4.3) {};

        \node[fill=black,circle,inner sep=1pt] (t11) at (2,6) {};
        \node[fill=black,circle,inner sep=1pt] (t12) at (3,6) {};
        \node[fill=black,circle,inner sep=1pt] (t13) at (4,6) {};
        \node[fill=black,circle,inner sep=1pt] (t14) at (5,6) {};

        \node[fill=black,circle,inner sep=1pt] (t15) at (7-1,5.7) {};
        \node[fill=black,circle,inner sep=1pt] (t16) at (8-1,5.7) {};
        \node[fill=black,circle,inner sep=1pt] (t17) at (9-1,5.7) {};

        \node[fill=black,circle,inner sep=1pt] (t18) at (7-1,6.3) {};
        \node[fill=black,circle,inner sep=1pt] (t19) at (8-1,6.3) {};
        \node[fill=black,circle,inner sep=1pt] (t110) at (9-1,6.3) {};

        \node[fill=black,circle,inner sep=1pt] (t21) at (2,2) {};
        \node[fill=black,circle,inner sep=1pt] (t22) at (3,2) {};
        \node[fill=black,circle,inner sep=1pt] (t23) at (4,2) {};
        \node[fill=black,circle,inner sep=1pt] (t24) at (5,2) {};

        \node[fill=black,circle,inner sep=1pt] (t25) at (7-1,1.7) {};
        \node[fill=black,circle,inner sep=1pt] (t26) at (8-1,1.7) {};
        \node[fill=black,circle,inner sep=1pt] (t27) at (9-1,1.7) {};

        \node[fill=black,circle,inner sep=1pt] (t28) at (7-1,2.3) {};
        \node[fill=black,circle,inner sep=1pt] (t29) at (8-1,2.3) {};
        \node[fill=black,circle,inner sep=1pt] (t210) at (9-1,2.3) {};

        \draw (t1)--(t11);
        \draw (t1)--(t21);

        \draw (t1)--(t3);
        \draw (t11)--(t12);
        \draw (t21)--(t22);

        \draw (t4)--(t5);
        \draw (t13)--(t14);
        \draw (t23)--(t24);

        \draw (t5)--(t6);
        \draw (t14)--(t15);
        \draw (t24)--(t25);
        \draw (t5)--(t9);
        \draw (t14)--(t18);
        \draw (t24)--(t28);

        \draw [dashed] (t6)--(t7);
        \draw [dashed] (t15)--(t16);
        \draw [dashed] (t25)--(t26);
        \draw [dashed] (t9)--(t10);
        \draw [dashed] (t18)--(t19);
        \draw [dashed] (t28)--(t29);

        \draw (t1)--(t3);
        \draw (t11)--(t12);
        \draw (t21)--(t22);
        \draw (t1)--(t3);
        \draw (t11)--(t12);
        \draw (t21)--(t22);
 
        \draw [dashed] (t3)--(t4);
        \draw [dashed] (t12)--(t13);
        \draw [dashed] (t22)--(t23);

        \draw (t7)--(t8);
        \draw (t10)--(t101);
        \draw (t26)--(t27);
        \draw (t29)--(t210);
        \draw (t16)--(t17);
        \draw (t19)--(t110);

        \node at (-0.2,4) {$\rho$};
        \node at (8-1,1.3) {$\underbrace{\hspace{5.8 em}}_{k_r \hbox{ vertices}}$};
        \node at (8-1,3.3) {$\underbrace{\hspace{5.8 em}}_{k_2 \hbox{ vertices}}$};
        \node at (8-1,5.3) {$\underbrace{\hspace{5.8 em}}_{k_1 \hbox{ vertices}}$};

        \node at (3,2.8) {$\cdots$};
        \node at (5,2.8) {$\cdots$};    
        \node at (7,2.8) {$\cdots$};

        \end{tikzpicture}
        \end{tabular}
\caption{The structure of a tree $T$ with height $h$ and order $n$ which minimizes the number of root-containing subtrees.}\label{fig:ex1_n}
\end{figure} 

Here we present several lemmas regarding the characteristics of an optimal tree. These all work toward the proof of Theorem~\ref{thm:optimal_summary}.

\begin{lemma}\label{lem:sub1}
In any optimal tree $T$, for any $v\in V(T)$, $T(v)$ minimizes the number of root-containing subtrees among  all rooted trees of the same order and height at most $h-h_T(v)$.
\end{lemma}
 
\begin{proof} 
Let $T$ be an optimal tree. 
Suppose, for contradiction, that there is a vertex $v$ for which $T(v)$ does not satisfy the lemma. In other words, there is a tree $T'(v)$, which is rooted at $v$, has the same order as $T(v)$, and has
$$ h(T'(v)) \leq h-h_T(v) \hbox{ and } F_{T'(v)}(v) < F_{T(v)}(v) . $$
Let $T'$ be the tree obtained from $T$ by replacing $T(v)$ with $T'(v)$. Then $T$ and $T'$ have the same number of subtrees containing $\rho$ but not $v$. Define $T^*:=T-(T(v)- \{v\})$ and let $F_{T^*}(\rho,v)$ be the number of subtrees of $T^*$ that contain both $\rho$ and $v$. Because $T$ and $T'$ only differ in the descendants of $v$, we have 
$$F_{T}(\rho)-F_{T'}(\rho)=F_{T(v)}(v)F_{T^*}(\rho,v)-F_{T'(v)}(v)F_{T^*}(\rho,v) > 0, $$ a contradiction to the optimality of $T$.
\end{proof}

\begin{lemma}\label{lem:sub2}
The height of any leaf in an optimal tree is $h$.
\end{lemma}

\begin{proof}
If $n=h+1$, it is straightforward to see that the path rooted at one end is the optimal tree. In the case when $n>h+1$, some vertex must have at least 2 children. 
Suppose, for contradiction, that there is a leaf $v\in V(T)$ whose height is less than $h$.
Let $x$ be the closest ancestor (possibly the root) of $v$ that has at least two children. 
Let $y$ be a child of $x$ that is not on $P(x,v)$ and $z$ be the child of $x$ on $P(x, v)$. 

\begin{figure}[htbp]
\centering
\begin{tabular}{c}
    \begin{tikzpicture}[scale=1.2]
        \node[fill=black,circle,inner sep=1pt] (t1) at (0,0) {};
        \node[fill=black,circle,inner sep=1pt] (t2) at (1,1) {};
        \node[fill=black,circle,inner sep=1pt] (t3) at (2,.5) {};
        \node[fill=black,circle,inner sep=1pt] (t4) at (3 ,1.5) {};

        \node[fill=black,circle,inner sep=1pt] (t5) at (5,0) {};
        \node[fill=black,circle,inner sep=1pt] (t6) at (6,1) {};
        \node[fill=black,circle,inner sep=1pt] (t7) at (7, 1) {};
        \node[fill=black,circle,inner sep=1pt] (t8) at (8 ,2) {};

       \draw (t1)--(t2);
       \draw (t1)--(t3);
       \draw (t1)--(-.5,1)--(.5,1)--cycle;
       \node at(0,.7){$T_x$};
       \draw (t2)--(-.4+1,2)--(.4+1,2)--cycle;

       \draw (t5)--(t6);
       \draw (t6)--(t7);
       \draw (t5)--(-.5+5,1)--(.5+5,1)--cycle;
       \node at (5,.7){$T_x$};
       \draw (t6)--(-.4+1+5,2)--(.4+1+5,2)--cycle;

  \node at (2.5,1) {.};
  \node at(2.3,0.8) {.};
  \node at(2.7,1.2) {.};
  \node at(0,-.2){$x$};
  \node at(1, 0.8){$y$};
  \node at(2,.3){$z$};
  \node at(3, 1.3){$v$};
  \node at(1,-.8){$T(x)$};

  \node at (7.5,1.5) {.};
  \node at(7.3,1.3) {.};
  \node at(7.7,1.7) {.};
  \node at(5,-.2){$x$};
  \node at(6, 0.8){$y$};
  \node at(7,.8){$z$};
  \node at(8, 1.8){$v$};
  \node at(6,-.8){$T'(x)$};

 \end{tikzpicture}
 \end{tabular}
\caption{Trees $T(x)$ and $T'(x)$ from Lemma~\ref{lem:sub2}}\label{fig:example}
\end{figure}
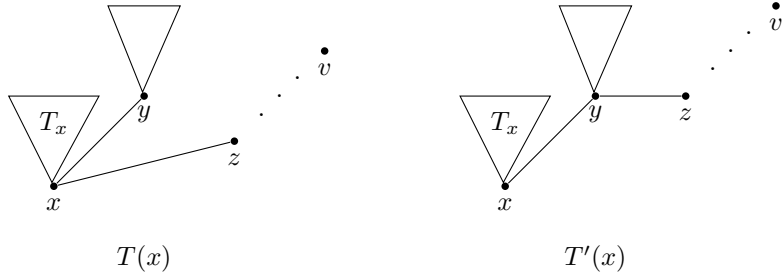 

Let $T_x$ be the component containing $x$ in $T-xy - xz$ and consider the tree
$$ T'(x) := T(x) - xz + yz $$ depicted in Figure~\ref{fig:example}.
Note that $T'(x)$ has the same order as $T(x)$ and has height no more than $h-h_T(x)$ because the height of $v$ in $T$ is less than $h$.

Counting the number of subtrees containing $x$ in each tree, we obtain the following equalities:
\begin{eqnarray*}
F_{T(x)}(x)&=&F_{T_x}(x)(1+d_{T(x)}(x,v))(1+F_{T(y)}(y)),\\
F_{T'(x)}(x)&=&F_{T_x}(x)\left[1+(1+d_{T(x)}(x,v))F_{T(y)}(y)\right]. \end{eqnarray*}
Together, these imply
$$F_{T(x)}(x)-F_{T'(x)}(x)=d_{T(x)}(x,v) F_{T_x}(x)> 0 .$$
Since $T$ was optimal, this contradicts Lemma~\ref{lem:sub1}.
\end{proof}

\begin{lemma}\label{lem:sub3}
Every optimal tree has one of the following two properties: 
\begin{itemize}
\item All non-root vertices have degree at most 3. 
\item All non-root vertices of height less than $h-1$ have degree at most 3. For any vertex $v$ of height $h-1$, $deg(v)\leq 4$. Further, if $deg(v)=4$, then the parent of $v$ must have degree 2 or be the root. 
\end{itemize}

\end{lemma}

\begin{proof}
As before, this proof proceeds by contradiction. Let $x$ be a non-root vertex in an optimal tree $T$ with degree at least 4. Say $y$, $z$, and $w$ are three children of $x$ and let $u$ be the parent of $x$. Denote by $T_u$ and $T_x$ the components containing $u$ and $x$ respectively in $T-ux - xy - xz - xw$. Without loss of generality, assume  
$$ F_{T(w)}(w) = \max\{ F_{T(y)}(y), F_{T(z)}(z), F_{T(w)}(w) \}. $$

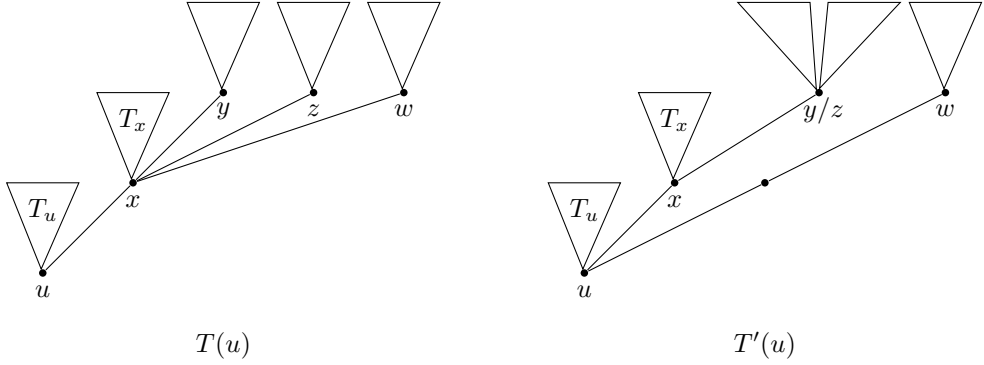
\begin{figure}[htbp]
\centering
\begin{tabular}{c}
    \begin{tikzpicture}[scale=1.2]
        \node[fill=black,circle,inner sep=1pt] (t1) at (0,0) {};
        \node[fill=black,circle,inner sep=1pt] (t2) at (1,1) {};
        \node[fill=black,circle,inner sep=1pt] (t3) at (2,2) {};
        \node[fill=black,circle,inner sep=1pt] (t4) at (3 ,2) {};
        \node[fill=black,circle,inner sep=1pt] (t5) at (4,2) {};

        \node[fill=black,circle,inner sep=1pt] (t6) at (6,0) {};
        \node[fill=black,circle,inner sep=1pt] (t7) at (7, 1) {};
        \node[fill=black,circle,inner sep=1pt] (t8) at (8 ,1) {};
        \node[fill=black,circle,inner sep=1pt] (t9) at (8.6 ,2) {};
        \node[fill=black,circle,inner sep=1pt] (t10) at (10 ,2) {};

       \draw (t1)--(t2);
       \draw (t2)--(t3);
       \draw (t2)--(t4);
       \draw (t2)--(t5);
       \draw (t1)--(-.4,1)--(.4,1)--cycle;
       \draw (t2)--(-.4+1,2)--(.4+1,2)--cycle;
       \draw (t3)--(-.4+2,3)--(.4+2,3)--cycle;
       \draw (t4)--(-.4+3,3)--(.4+3,3)--cycle;
       \draw (t5)--(-.4+4,3)--(.4+4,3)--cycle;

       \draw (t6)--(t7);
       \draw (t6)--(t8);
       \draw (t7)--(t9);
       \draw (t8)--(t10);
       \draw (t6)--(-.4+6,1)--(.4+6,1)--cycle;
       \draw (t7)--(-.4+7,2)--(.4+7,2)--cycle;
       \draw (t9)--(7.7,3)--(8.5,3)--cycle;
       \draw (t9)--(8.7,3)--(9.5,3)--cycle;
       \draw (t10)--(-.4+10,3)--(.4+10,3)--cycle;

  \node at(0,-.2){$u$};
  \node at(0,.7){$T_u$};
  \node at(1, 0.8){$x$};
  \node at(1,1.7){$T_x$};
  \node at(2,1.8){$y$};
  \node at(3, 1.8){$z$};
  \node at(4,1.8){$w$};
  \node at(2,-.8){$T(u)$};

  \node at(6,-.2){$u$};
  \node at(6,.7){$T_u$};
  \node at(7, 0.8){$x$};
  \node at(7,1.7){$T_x$};
  \node at(8.65,1.8){$y/z$};
  \node at(10, 1.8){$w$};
  \node at(8,-.8){$T'(u)$};

 \end{tikzpicture}
 \end{tabular}
\caption{Trees $T(u)$ and $T'(u)$ in the proof of Lemma~\ref{lem:sub3}.}\label{fig:example'}
\end{figure}

Now consider the tree $T'(u)$ obtained from $T(u)$ by removing the edges $xz$ and $xw$, inserting a path of length 2 between $u$ and $w$, while identifying the vertices $y$ and $z$ (Figure~\ref{fig:example'}). Note that $T'(u)$ has the same height and order as $T(u)$. Counting the number of subtrees containing $u$ in each, we find
\begin{eqnarray*}
F_{T(u)}(u)&=&F_{T_u}(u)\left[ 1+F_{T_x}(x)\left(1+F_{T(y)}(y)\right)\left(1+F_{T(z)}(z)\right)\left(1+F_{T(w)}(w)\right)\right],\\
F_{T'(u)}(u)&=&F_{T_u}(u)\left(2+F_{T(w)}(w)\right)\left[ 1+F_{T_x}(x)\left(1+F_{T(y)}(y)F_{T(z)}(z)\right)\right].
\end{eqnarray*}
Together, these imply the following
\begin{align}
F_{T(u)}(u)-F_{T'(u)}(u) & =
F_{T_u}(u)\left[F_{T_x}(x)F_{T(y)}(y)\left(F_{T(w)}(w)-F_{T(z)}(z)\right) \right. \nonumber\\
& \hspace{3 em} +F_{T_x}(x)\left(F_{T(y)}(y)-1\right)  \nonumber\\
& \hspace{3 em} \left. +\left(F_{T(w)}(w)+1\right)\left(F_{T_x}(x)F_{T(z)}(z)-1\right) \right] \nonumber\\
&\geq 0 . \label{ineq1}
\end{align}

Because $T$ is an optimal tree, $T(u)$ is an optimal tree by Lemma~\ref{lem:sub1}. Therefore \eqref{ineq1} must be equality. 
Note that for any tree $H$ and vertex $a\in V(H)$, $F_H(v)\geq 1$ because the subtree containing only the vertex $v$ will be counted. 
Therefore, equality holds in \eqref{ineq1} exactly when $F_{T_x}(x)=F_{T(y)}(y)=F_{T(z)}(z)=F_{T(w)}(w)=1$, or equivalently, $deg(x)=4$ and $y,z,w$ are all leaves so $x$ has height $h-1$ in $T$. Create $T'$ from $T$ by replacing $T(u)$ with $T'(u)$. Because \eqref{ineq1} is equality, $F_{T(u)}(u) = F_{T'(u)}(u)$. Therefore $T'$ is also an optimal tree. 

In $T'$, $deg_{T'}(x)=3$ but $deg_{T'}(u) = deg_T(u) + 1$. Observe $u$ has height $h-2$ in $T'$. If $u$ is not the root of $T'$ and $deg_{T'}(u)\geq 4$, then we can repeat the argument for optimal tree $T'$ and vertex $u$ having degree at least 4. Because the height of $u$ is $h-2$, we will find a contradiction in the step which parallels \eqref{ineq1}. Therefore $deg_{T'}(u) \leq 3$ which implies $deg_T(u) \leq 2$. Since $u$ is not the root of $T$, we can conclude $deg_T(u)=2$ as stated in the lemma.  
\end{proof}

In the proof of Lemma~\ref{lem:sub3}, in the case where $deg_T(x)=4$, we created another optimal tree $T'$ where $deg_{T'}(x)=3$ and no other degree 4 vertices where created. Hence, if an optimal tree has multiple degree 4 vertices of height $h-1$, we can repeat this procedure to obtain an optimal $T'$ with all vertices of degree at most 3. This establishes the following observation. 

\begin{obs}
There is an optimal tree in which all non-root vertices have degree at most 3. 
\end{obs}

We now shift our attention to the structures of $T_i$ for $1\leq i \leq k$.

\begin{lemma}\label{lem:sub4}
In an optimal tree $T$, each subtree $T_i\cup\{\rho\}$ falls into one of the following three categories:
\begin{itemize}
\item There is at most one non-root vertex with degree $3$.
\item All non-root vertices of height at most $h-3$ have degree $2$, the vertex of height $h-2$ has degree $3$, and exactly one of its children has degree $3$. 
\item All non-root vertices of height at most $h-2$ have degree $2$ and the vertex of height $h-1$ has degree $4$. 
\end{itemize}
\end{lemma}

\begin{proof}
We prove this in two pieces, considering the alternatives from Lemma \ref{lem:sub3} separately.  We start with the optimal trees in which all vertices have degree at most 3. 

For contradiction, suppose there exists a $T_i\cup\{\rho\}$ with at least two non-root vertices of degree $3$. Let $v$ be a degree 3 vertex of greatest height in $T_i$ and let $u, w$ be the two children of $v$.  Let $z$ be the closest ancestor of $v$ such that $deg_{T_i}(z)=3$, $z$ has parent $x$, and $z$ has child $y \notin V(P(z,v))$. Let $\ell_1$ denote the distance from $v$ to a leaf in $T_i$ and $\ell_2$ the length of $P(v, z)$. Let $T_x$ denote the component containing $x$ in $T(x)-xz$ (Figure~\ref{fig:ex}). 

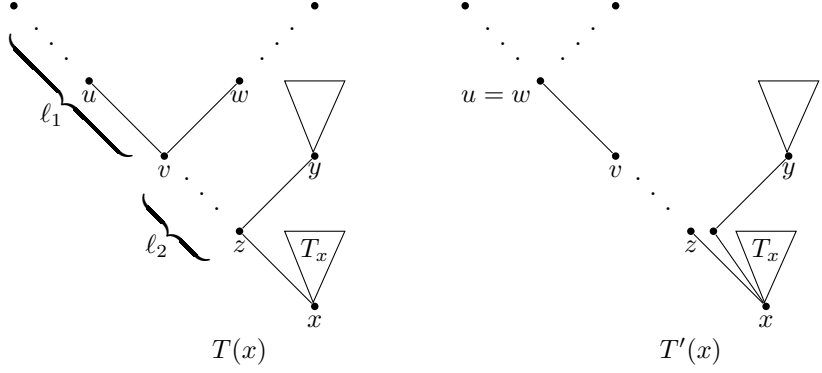
\begin{figure}[htbp]
\centering
\begin{tabular}{c}
    \begin{tikzpicture}[scale=1]
        \node[fill=black,circle,inner sep=1pt] (t1) at (0,4) {};
        \node[fill=black,circle,inner sep=1pt] (t2) at (1,3) {};
        \node[fill=black,circle,inner sep=1pt] (t3) at (2,2) {};
        \node[fill=black,circle,inner sep=1pt] (t4) at (3 ,3) {};
        \node[fill=black,circle,inner sep=1pt] (t5) at (4,4) {};
        \node[fill=black,circle,inner sep=1pt] (t6) at (3,1) {};
        \node[fill=black,circle,inner sep=1pt] (t7) at (4, 2) {};
        \node[fill=black,circle,inner sep=1pt] (t8) at (4 ,0) {};

        \node[fill=black,circle,inner sep=1pt] (t9) at (6,4) {};
        \node[fill=black,circle,inner sep=1pt] (t10) at (7,3) {};
        \node[fill=black,circle,inner sep=1pt] (t11) at (8,4) {};
        \node[fill=black,circle,inner sep=1pt] (t12) at (8, 2) {};
        \node[fill=black,circle,inner sep=1pt] (t13) at (9 ,1) {};
        \node[fill=black,circle,inner sep=1pt] (t14) at (10 ,0) {};
        \node[fill=black,circle,inner sep=1pt] (t15) at (9.3,1) {};
        \node[fill=black,circle,inner sep=1pt] (t16) at (10.3,2) {};

       \draw (t2)--(t3);
       \draw (t3)--(t4);
       \draw (t6)--(t7);
       \draw (t6)--(t8);
       \draw (t7)--(-.4+4,3)--(.4+4,3)--cycle;
       \draw (t8)--(-.4+4,1)--(.4+4,1)--cycle;
       \node at(4,.75){$T_x$};

       \draw (t10)--(t12);
       \draw (t13)--(t14);
       \draw (t14)--(t15);
       \draw (t15)--(t16);
       \draw (t14)--(-.4+10,1)--(.4+10,1)--cycle;
       \node at(10,.75){$T_x$};
       \draw (t16)--(-.4+10.3,3)--(.4+10.3,3)--cycle;

  \node at(0.3,3.7){.};  \node at(0.5,3.5){.};  \node at(0.7,3.3){.};
  \node at(3.3,3.3){.};  \node at(3.5,3.5){.};  \node at(3.7,3.7){.};  
  \node at(2.3,1.7){.};  \node at(2.5,1.5){.};  \node at(2.7,1.3){.};  
  \node at(1,2.8){$u$};
  \node [rotate=-45] at(.7,2.75){$\substack{\underbrace{\hspace{6.5 em}}}$};
  \node at(.5,2.55){$\ell_1$};
  \node at(3, 2.8){$w$};
  \node at(2,1.8){$v$};
  \node at(3,0.8){$z$};
  \node [rotate=-45] at(2.1,1){$\substack{\underbrace{\hspace{3.5 em}}}$};  
  \node  at(1.9,.8){$\ell_2$};
  \node at(4,-.2){$x$};
  \node at(4,1.8){$y$};
  \node at(3,-.6){$T(x)$};

  \node at(6.3,3.7){.};  \node at(6.5,3.5){.};  \node at(6.7,3.3){.};
  \node at(7.3,3.3){.};  \node at(7.5,3.5){.};  \node at(7.7,3.7){.};  
  \node at(8.3,1.7){.};  \node at(8.5,1.5){.};  \node at(8.7,1.3){.};  
  \node [left] at(7,2.8){$u=w$};
  \node at(8,1.8){$v$};
  \node at(9,0.8){$z$};
  \node at(10,-.2){$x$};
  \node at(10.3,1.8){$y$};
  \node at(9,-.6){$T'(x)$};

 \end{tikzpicture}
 \end{tabular}
\caption{Transforming $T(x)$ into $T'(x)$ when $deg_T(v)=3$ in the proof of Lemma~\ref{lem:sub4}.}\label{fig:ex}
\end{figure}

Create a new tree $T'(x)$ from $T(x)$ by removing the edges $vw$ and $zy$, inserting a length 2 path between $x$ and $y$, and identifying $u$ and $w$ (Figure~\ref{fig:ex}). Note that $T'(x)$ has the same height and order as $T(x)$. The number of subtrees containing $x$ in each is
\begin{align*}
F_{T(x)}(x) &=F_{T_x}(x)\left[1+(1+F_{T(y)}(y))[\ell_2+(\ell_1+1)^2]\right], \\
F_{T'(x)}(x) & =F_{T_x}(x)(2+F_{T(y)}(y))\left(\ell_2+2+\ell_1^2\right) . \end{align*}
By Lemma~\ref{lem:sub2}, the height of each leaf in $T$ is $h$, hence $V(T(y)) \geq \ell_1+\ell_2$. Now we have
\begin{align}
F_{T(x)}(x)-F_{T'(x)}(x)&=F_{T_x}(x)\left[(1+F_{T(y)}(y))(2\ell_1-1)-(\ell_1^2+\ell_2+1)]\right] \nonumber \\
& \geq F_{T_x}(x)\left[(1+\ell_1 + \ell_2)(2\ell_1-1)-(\ell_1^2+\ell_2+1)]\right]  \label{bound} \\
& = F_{T_x}(x)(\ell_1^2 + 2\ell_1\ell_2 + \ell_1 -2\ell_2 - 2) \nonumber \\
&\geq 0 . \label{zero} \end{align}
When either \eqref{bound} or \eqref{zero} is strict inequality, we have a contradiction to the optimality of $T$. Equality holds exactly when $\ell_1 = \ell_2 = 1$ and $|V(T(y))| = \ell_1 + \ell_2$. In other words, $T(y)$ is a single path on two vertices with $y$ having height $h-1$. Since $T'$ (constructed from $T$ by replacing $T(x)$ with $T'(x)$) is an optimal tree, if $x$ is not the root, $deg_{T'}(x) \leq 3$ since $x$ has height $h-3$. Therefore $deg_T(x)\leq 2$ as described in the second property of the lemma. 

If $T$ falls into the second category listed in Lemma~\ref{lem:sub3}, then consider a subtree $T_i$ with a vertex $v$ of degree $4$ at height $h-1$. We will show that all other non-root vertices in $T_i\cup\{\rho\}$ must have degree 2.
Suppose to the contrary that $v$ has an ancestor $z$ of degree 3. (In this way, we are able to simultaneously handle the case when there are two vertices of degree $4$ in $T_i\cup \{\rho\}$ because they would have to share a common ancestor of degree 3.)
Label the vertices as before with $s$ being the third child of $v$ (Figure \ref{fig:exdeg4}). 

Create $T'(x)$ by altering $T(x)$ in a manner similar to that described above.  Define \[T'(x)=T(x)-wv-yz+xw+wy\] as shown in Figure \ref{fig:exdeg4}.

\begin{figure}[htbp]
\centering
\begin{tabular}{c}
    \begin{tikzpicture}[scale=1.2]
        \node[fill=black,circle,inner sep=1pt] (t2) at (1,3) {};
        \node[fill=black,circle,inner sep=1pt] (t3) at (2,2) {};
        \node[fill=black,circle,inner sep=1pt] (t4) at (3 ,3) {};
        \node[fill=black,circle,inner sep=1pt] (t5) at (2,3) {};
        \node[fill=black,circle,inner sep=1pt] (t6) at (3,1) {};
        \node[fill=black,circle,inner sep=1pt] (t7) at (4, 2) {};
        \node[fill=black,circle,inner sep=1pt] (t8) at (4 ,0) {};

        \node[fill=black,circle,inner sep=1pt] (t10) at (7,3) {};
        \node[fill=black,circle,inner sep=1pt] (t11) at (9,3) {};
        \node[fill=black,circle,inner sep=1pt] (t12) at (8, 2) {};
        \node[fill=black,circle,inner sep=1pt] (t13) at (9 ,1) {};
        \node[fill=black,circle,inner sep=1pt] (t14) at (10 ,0) {};
        \node[fill=black,circle,inner sep=1pt] (t15) at (9.3,1) {};
        \node[fill=black,circle,inner sep=1pt] (t16) at (10.3,2) {};

       \draw (t2)--(t3);
       \draw (t3)--(t4);
       \draw (t5)--(t3);
       \draw (t6)--(t7);
       \draw (t6)--(t8);
       \draw (t7)--(-.4+4,3)--(.4+4,3)--cycle;
       \draw (t8)--(-.4+4,1)--(.4+4,1)--cycle;
       \node at(4,.7){$T_x$};
       \draw (t10)--(t12);
       \draw (t11)--(t12);
       \draw (t13)--(t14);
       \draw (t14)--(t15);
       \draw (t15)--(t16);
       \draw (t14)--(-.4+10,1)--(.4+10,1)--cycle;
       \node at(10,.7){$T_x$};
       \draw (t16)--(-.4+10.3,3)--(.4+10.3,3)--cycle;

  \node at(2.3,1.7){.};  \node at(2.5,1.5){.};  \node at(2.7,1.3){.};  
  \node at(1,2.8){$u$};
  \node at(1.8, 2.8){$w$};
  \node at(3, 2.8){$s$};
  \node at(2,1.8){$v$};
  \node at(3,0.8){$z$};
  \node [rotate=-45] at(2.2, 1.2){$\substack{\underbrace{\hspace{3.5 em}}}$};  
  \node  at(2.1,.9){$\ell_2$};
  \node at(4,-.2){$x$};
  \node at(4,1.8){$y$};
  \node at(3,-.6){$T(x)$};

  \node at(8.3,1.7){.};  \node at(8.5,1.5){.};  \node at(8.7,1.3){.};  
  \node at(7,2.8){$u$};
  \node at(9,2.8){$s$};
  \node at(8,1.8){$v$};
  \node at(9,0.8){$z$};
  \node at(9.3,1.2){$w$};
  \node at(10,-.2){$x$};
  \node at(10.3,1.8){$y$};
  \node at(9,-.6){$T'(x)$};

 \end{tikzpicture}
 \end{tabular}
\caption{Transforming $T(x)$ into $T'(x)$ when $v$ has degree $4$ in the proof of Lemma~\ref{lem:sub4}.}\label{fig:exdeg4}
\end{figure}
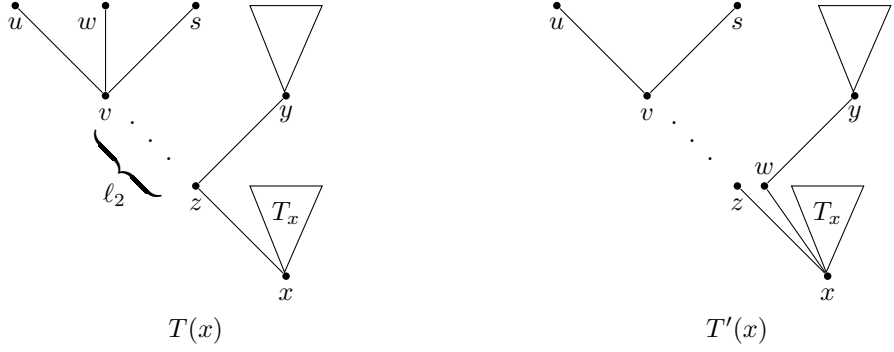

Let $\ell_2$ be the distance from $z$ to $v$ in $T(x)$. Because all leaves have height $h$, $F_{T(y)}(y)\geq \ell_2+1$ which is tight when $T(y)$ is a path. Now if we calculate $F_{T(x)}(x)$ and  $F_{T'(x)}(x)$ exactly and take their difference, we find
\begin{eqnarray*}
F_{T(x)}(x) &=& F_{T_x}(x) \left(1+ (1+F_{T(y)}(y))(\ell_2+8)\right)\\
F_{T'(x)}(x) & =& F_{T_x}(x)  ( 2+ F_{T(y)}(y) ) ( \ell_2 + 5)\\
F_{T(x)}(x) - F_{T'(x)}(x) &=& F_{T_x}(x)  \left( 3F_{T(y)}(y) - \ell_2 -1 \right) \\
&\geq & F_{T_x}(x)  \left( 3(\ell_2+1) - \ell_2 -1 \right) \\
&= & F_{T_x}(x)  \left( 2\ell_2+2 \right)\\
&>& 0.
\end{eqnarray*}
This contradicts our choice of $T$. Thus $T_i \cup \{\rho\}$ can have at most one vertex of degree 4 and all other non-root vertices must have degree 2 as described in the third property of the lemma. 
\end{proof}

Once again, it is useful to note that for each of the optimal trees described in the second two properties of Lemma \ref{lem:sub4}, the proof supplies $T'$ analogues which have the same number of root-containing subtrees and yet fall under the first property description in Lemma \ref{lem:sub4}. This gives the following observation.

\begin{obs} 
There is an optimal tree with each $T_i\cup\{\rho\}$ having at most one non-root vertex of degree 3. 
\label{T_i split}
\end{obs} 

%
Observation~\ref{T_i split} establishes that each $T_i$ is a $k_i$-split for some integer $k_i$ with $0\leq k_i \leq h-1$. To minimize some notation, we state the following structural observation. 

\begin{obs}
In an optimal tree $T$, the number of root-containing subtrees in a $k_i$-split together with root $\rho$ is
\[s_h(k_i): = h + k_i^2 +k_i +1.\] 
This definition also makes sense for the $0$-split (together with $\rho$), which has $h+1$ root-containing subtrees.  
\end{obs}

\begin{lemma}\label{lem:sub5}
Among the $T_i$ subtrees in an optimal tree, at most two of them can be 0-splits.
\end{lemma}
\begin{proof} 
Suppose, for contradiction, that  $T_i$, $T_j$ and $T_k$ are each 0-splits in an optimal tree. Consider $S: = T_i \cup T_j \cup T_k \cup \{ \rho \}$. 
Create $S'$ from $S$ by replacing $T_i$ with a 1-split, $T_j$ with an $(h-1)$-split and deleting $T_k$ (Figure~\ref{fig:ex_n}). 

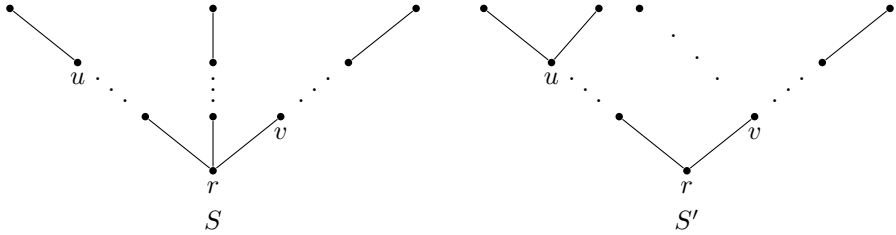
\begin{figure}[htbp]
\centering
\begin{tabular}{c}
    \begin{tikzpicture}[scale=0.9, y=.8cm]
        \node[fill=black,circle,inner sep=1pt] (t1) at (0,3) {};
        \node[fill=black,circle,inner sep=1pt] (t2) at (1,2) {};
        \node[fill=black,circle,inner sep=1pt] (t3) at (2,1) {};
        \node[fill=black,circle,inner sep=1pt] (t4) at (3 ,0) {};
        \node[fill=black,circle,inner sep=1pt] (t5) at (3,1) {};
        \node[fill=black,circle,inner sep=1pt] (t6) at (3,2) {};
        \node[fill=black,circle,inner sep=1pt] (t7) at (3, 3) {};
        \node[fill=black,circle,inner sep=1pt] (t8) at (4 ,1) {};
        \node[fill=black,circle,inner sep=1pt] (t9) at (5,2) {};
        \node[fill=black,circle,inner sep=1pt] (t10) at (6,3) {};

        \node[fill=black,circle,inner sep=1pt] (t11) at (7,3) {};
        \node[fill=black,circle,inner sep=1pt] (t12) at (8,2) {};
        \node[fill=black,circle,inner sep=1pt] (t13) at (8.7,3) {};
        \node[fill=black,circle,inner sep=1pt] (t14) at (9,1) {};
        \node[fill=black,circle,inner sep=1pt] (t15) at (10, 0) {};
        \node[fill=black,circle,inner sep=1pt] (t16) at (9.3 ,3) {};       
        \node[fill=black,circle,inner sep=1pt] (t17) at (11,1) {};
        \node[fill=black,circle,inner sep=1pt] (t18) at (12,2) {};
        \node[fill=black,circle,inner sep=1pt] (t19) at (13 ,3) {};

       \draw (t1)--(t2);
       \draw (t3)--(t4);
       \draw (t4)--(t5);
       \draw (t6)--(t7);
       \draw (t4)--(t8);
       \draw (t9)--(t10);
       \draw (t11)--(t12);
       \draw (t12)--(t13);
       \draw (t14)--(t15);
       \draw (t15)--(t17);
       \draw (t18)--(t19);

 \node at(1,1.7){$u$};
 \node at(8,1.7){$u$};
 \node at(4,0.7){$v$};
 \node at(11,0.7){$v$};

  \node at(3,-.3){$r$};
  \node at(10,-.3){$r$};

 \node at(1.3,1.7){.};  \node at(1.5,1.5){.};  \node at(1.7,1.3){.};
 \node at(3,1.7){.};  \node at(3,1.5){.};  \node at(3,1.3){.};
 \node at(4.3,1.3){.};  \node at(4.5,1.5){.};  \node at(4.7,1.7){.};
 \node at(8.3,1.7){.};  \node at(8.5,1.5){.};  \node at(8.7,1.3){.};
 \node at(11.3,1.3){.};  \node at(11.5,1.5){.};  \node at(11.7,1.7){.};
 \node at(9.8,2.5){.};  \node at(10.15,2.1){.};  \node at(10.45,1.7){.};

  \node at(3,-.9){$S$};
  \node at(10,-.9){$S'$};

 \end{tikzpicture}
 \end{tabular}
\caption{Trees $S$ and $S'$ from the proof of Lemma~\ref{lem:sub5}}\label{fig:ex_n}
\end{figure}

The difference in the number of subtrees is
\begin{align*}
F_{S}(\rho)-F_{S'}(\rho)
&= (s_h(0))^3 - s_h(1)s_h(h-1)\\
&= (h+1)^3 - (h+3)(2h+(h-1)^2)\\
& = 2(h-1)\\
&> 0. 
\end{align*}
This contradicts the optimality of $T$ because the tree obtained from $T$ by replacing $S$ with $S'$ has fewer root-containing subtrees than $T$. 
\end{proof}

\begin{lemma}
If some $T_i$ is a $0$-split, then for each $j\neq i$, $T_j$ is either a $0$-split or a $1$-split.  \label{lem:sub6}
\end{lemma}
\begin{proof}
Suppose instead that $T_i $ is a $0$-split and $T_j$ is a $k_j$-split where $1<k_j \leq h-1$. 
Let $S$ be the tree induced by $T_i$, $T_j$ and $r$. Construct $S'$ from $S$ by replacing $T_i$ with a 1-split and replacing $T_j$ with a $(k_j-1)$-split (Figure~\ref{fig:ex'_n}). 

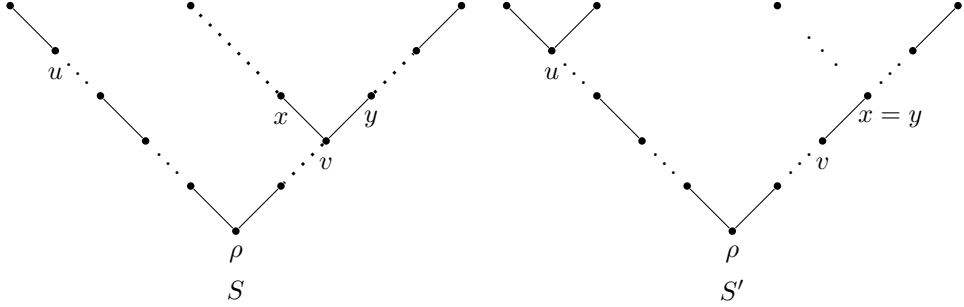
\begin{figure}[htbp]
\centering
\begin{tabular}{c}
    \begin{tikzpicture}[scale=0.6]
        \node[fill=black,circle,inner sep=1pt] (t1) at (0,5) {};
        \node[fill=black,circle,inner sep=1pt] (t2) at (1,4) {};
        \node[fill=black,circle,inner sep=1pt] (t3) at (2,3) {};
        \node[fill=black,circle,inner sep=1pt] (t4) at (3 ,2) {};
        \node[fill=black,circle,inner sep=1pt] (t5) at (4,1) {};
        \node[fill=black,circle,inner sep=1pt] (t6) at (5,0) {};
        \node[fill=black,circle,inner sep=1pt] (t7) at (6, 1) {};
        \node[fill=black,circle,inner sep=1pt] (t8) at (7,2) {};
        \node[fill=black,circle,inner sep=1pt] (t9) at (8,3) {};
        \node[fill=black,circle,inner sep=1pt] (t9a) at (6,3) {};
        \node[fill=black,circle,inner sep=1pt] (t10) at (9,4) {};
        \node[fill=black,circle,inner sep=1pt] (t11) at (10,5) {};

        \node[fill=black,circle,inner sep=1pt] (t12) at (11,5) {};
        \node[fill=black,circle,inner sep=1pt] (t13) at (12,4) {};
        \node[fill=black,circle,inner sep=1pt] (t14) at (13,3) {};
        \node[fill=black,circle,inner sep=1pt] (t15) at (14, 2) {};
        \node[fill=black,circle,inner sep=1pt] (t16) at (15 ,1) {};       
        \node[fill=black,circle,inner sep=1pt] (t17) at (16,0) {};
        \node[fill=black,circle,inner sep=1pt] (t18) at (17,1) {};
        \node[fill=black,circle,inner sep=1pt] (t19) at (18,2) {};
        \node[fill=black,circle,inner sep=1pt] (t20) at (19,3) {};
        \node[fill=black,circle,inner sep=1pt] (t21) at (20,4) {};
        \node[fill=black,circle,inner sep=1pt] (t22) at (21,5) {};

        \node[fill=black,circle,inner sep=1pt] (t23) at (4,5) {};
        \node[fill=black,circle,inner sep=1pt] (t24) at (13, 5) {};
        \node[fill=black,circle,inner sep=1pt] (t25) at (17 ,5) {};

       \draw (t1)--(t2);
       \draw (t3)--(t4);
       \draw (t5)--(t6);
       \draw (t6)--(t7);
       \draw (t8)--(t9);
       \draw (t8)--(t9a);
       \draw (t10)--(t11);

       \draw (t12)--(t13);
       \draw (t14)--(t15);
       \draw (t16)--(t17);
       \draw (t17)--(t18);
       \draw (t19)--(t20);
       \draw (t21)--(t22);
       \draw (t13)--(t24);
       \draw [loosely dotted, line width=1.1pt] (t9a)--(t23);
       \draw [loosely dotted, line width=1.1pt] (t9)--(t10);
       \draw [loosely dotted, line width=1.1pt] (t8)--(t7);

 \node at(17.7,4.3){.};  \node at(18,4){.};  \node at(18.3,3.7){.};
 \node at(1.3,3.7){.};  \node at(1.5,3.5){.};  \node at(1.7,3.3){.};
 \node at(3.3,1.7){.};  \node at(3.5,1.5){.};  \node at(3.7,1.3){.};
 \node at(12.3,3.7){.};  \node at(12.5,3.5){.};  \node at(12.7,3.3){.};
 \node at(14.3,1.7){.};  \node at(14.5,1.5){.};  \node at(14.7,1.3){.};
 \node at(17.3,1.3){.};  \node at(17.5,1.5){.};  \node at(17.7,1.7){.};
 \node at(19.3,3.3){.};  \node at(19.5,3.5){.};  \node at(19.7,3.7){.};

\node at(18,1.5){$v$};
\node at(7,1.5){$v$};
\node at(19.5,2.5){$x=y$};
\node at(8,2.5){$y$};
\node at (6, 2.5) {$x$};

\node at(1,3.5){$u$};
\node at(12,3.5){$u$};

\node at(5,-.5){$\rho$};
\node at(16,-.5){$\rho$};
\node at(5,-1.3){$S$};
\node at(16,-1.3){$S'$};

 \end{tikzpicture}
 \end{tabular}
\caption{Tree $S$ and tree $S'$, which is the result from identifying $x$ and $y$, from the proof of Lemma~\ref{lem:sub6}.}\label{fig:ex'_n}
\end{figure}

Note that $S'$ has the same height and order as $S$, and
\begin{align*}
F_{S}(\rho)-F_{S'}(\rho) & =
s_h(0)s_h(k_j) - s_h(1)s_h(k_j-1)\\
&= (h+1)\left[h+k_j^2 + k_j +1\right] - (h+3)\left[h+(k_j - 1)^2 +k_j\right] \\
&= 2\left[(k_j-1)(h - k_j) +(k_j- 1)\right] \\
& > 0 \tag{for $k_j>1$}.
\end{align*}
This contradicts the optimality of $T$ because the tree obtaining from $T$ by replacing $S$ with $S'$ has fewer root-containing subtrees than $T$. 
\end{proof}

\begin{lemma}
A rooted tree $T$ is not optimal if for any $T_i$ ($k_i$-split) and $T_j$ ($k_j$-split), we have  $k_i(1+k_j) > h+1$ for $1\leq k_i\leq k_j\leq h-1$.
\label{lem:sub7}
\end{lemma}
\begin{proof}
Define $T_2$ be the subtree of $T$ which consists of the root $\rho$ together with $T_i$ and $T_j$.

Construct $T'_2$ from $T_2$ by replacing $T_i$ with a $(k_i-1)$-split and replacing $T_j$ with a $(k_j+1)$-split. This construction is well-defined because $1\leq k_i$ and $k_j\leq h-1$. 

It is easy to see that $T'_2$ has the same height and order as $T_2$. We have
\[F_{T_2}(\rho)=s_h(k_i)s_h(k_j)\quad \quad \text{ and } \quad \quad F_{T'_2}(\rho)=s_h(k_i-1)s_h(k_j+1).\]
Since $k_i \leq k_j $ and $k_i(1+k_j)> h+1$, we have
$$ F_{T_2}(\rho)-F_{T'_2}(\rho)= -2(k_i-k_j-1)(k_i+k_ik_j-h-1)>0 , $$
which contradicts the optimality of $T$. 
\end{proof}

By reversing the roles of $k_i$ and $k_j$ in the previous lemma, we obtain the following corollary. 

\begin{cor}
A rooted tree $T$ is not optimal if for any $T_i$, which is a $k_i$-split, and $T_j$, which is a $k_j$-split, we have $k_j(1+k_i)<h+1$ and $k_i<k_j-1$.
\label{cor:sub8}
\end{cor}

\begin{cor}
Fix an optimal tree $T$ in which each $T_i$ is a $k_i$-split with $k_i\geq k_{i+1}$. If $k_1> \sqrt{h+\frac{5}{4}}-\frac{1}{2}$, then $k_i < \sqrt{h+\frac{5}{4}}-\frac{1}{2}$ for each $i\geq 2$.
\label{one big}
\end{cor}

\begin{proof}
Let $T$ be an optimal tree, as described in the corollary, with $k_1> \sqrt{h+\frac{5}{4}}-\frac{1}{2}$. For contradiction, suppose $k_2\geq  \sqrt{h+\frac{5}{4}}-\frac{1}{2}$. Observe 
\begin{align*}
k_2(k_1+1)
& >  \left(\sqrt{h+\frac{5}{4}}-\frac{1}{2}\right) \left(\sqrt{h+\frac{5}{4}}+\frac{1}{2}\right)\\
&= h+1.
\end{align*}
However, this contradicts the statement of Lemma~\ref{lem:sub7}. Therefore the corollary holds. 
\end{proof}

\begin{cor}
Fix an optimal tree $T$ in which each $T_i$ is a $k_i$-split. For any pair $\{k_i, k_j\}$ with $k_i,k_j\leq \sqrt{h+\frac{5}{4}}-\frac{1}{2}$, we can conclude $|k_i - k_j|\leq 1$. 
\label{even distrib}
\end{cor}

\begin{proof}
Suppose $k_i,k_j \leq \sqrt{h+\frac{5}{4}}-\frac{1}{2}$ with $k_j\geq k_i+2$. Observe
\begin{align*}
k_j(1+k_i) 
&\leq  k_j(k_j-1)\\
&\leq  \left(\sqrt{h+\frac{5}{4}}-\frac{1}{2}\right) \left(\sqrt{h+\frac{5}{4}}-\frac{3}{2}\right)\\
&= h+2 - 2\sqrt{h+\frac{5}{4}} \\
&< h+1.
\end{align*}
This contradicts Corollary~\ref{cor:sub8}, finishing the proof. 
\end{proof}

Going further in this direction, we have the following even more specific statements.

\begin{lemma}
Suppose $k_1 \geq k_2 \geq \ldots \geq k_k$. 
If $k_1> \left \lceil \sqrt{h+\frac{5}{4}}-\frac{1}{2}\right\rceil$, then for each $i>1$, 
\begin{equation}\frac{h+1}{k_1}-1 \leq k_i \leq \frac{h+1}{k_1+1}.\label{k_i bounds}\end{equation}
In particular, $k_2 = k_3 = \ldots =k_k = \left\lfloor \frac{h+1}{k_1+1} \right\rfloor$ provided $\left\lfloor \frac{h+1}{k_1+1} \right\rfloor \geq \frac{h+1}{k_1}-1$. 
\end{lemma}

\begin{proof}
Since $k_1 > \sqrt{h+\frac{5}{4}}-\frac{1}{2}$, Corollary~\ref{one big} implies $k_2< \sqrt{h+\frac{5}{4}}-\frac{1}{2}$. 
Since $T$ is optimal, Lemma~\ref{lem:sub7} yields $k_i(1+k_1) \leq k_2(1+k_1)\leq h+1$.  Thus 
\[k_i \leq \frac{h+1}{k_1+1}.\]

Because $k_1 > \left \lceil \sqrt{h+\frac{5}{4}}-\frac{1}{2}\right\rceil$, then necessarily $k_i\leq k_2< k_1-1$. Corollary~\ref{cor:sub8} gives 
\[k_1(1+k_i) \geq h+1.\] 
This is equivalent to
\[k_i \geq \frac{h+1}{k_1}-1.\]
\end{proof}

We have established the three cases detailed in Theorem~\ref{thm:optimal_summary}, aside from the further information when $n \geq 5h^2$. In the rest of this section, we examine the number of $T_i$ subtrees which can be a $k_i$-splits in an optimal tree. This will shed light on the values of $k_i$ in an optimal tree with a large number of vertices, compared to its height, and also lends insight into the degree of the root vertex.  First we prove a technical lemma.

\begin{lemma}
For fixed $h,n\in \mathbb{Z}$, let $T$ be an optimal tree with $n$ vertices, height $h$, and where $T_i$ is a $k_i$-split. Fix $t\in \mathbb{R}$, $t\geq 2$ which satisfies the inequality $h^{1/(t+1)} >\ln (6h)$.
For $x\in \left[ h^{1/(t+1)}, h^{1/t} \right]$ with $n\geq (h+x)(h+x-1) + 1$, then \[|\{i: k_i =x\}|<h+x-1.\]
\label{lem:technical_ki}
\end{lemma}

\begin{proof}
Let $T$ be a tree with root degree $r$ and each $T_i$ is a $k_i$-split. 
Suppose for contradiction that $k_1 = k_2 = \ldots = k_{h+x-1} = x$ (where the $k_i$ values are not necessarily in non-increasing order). 

 Let $H$ be the subtree induced by $T_1, \ldots, T_{h+x-1}$ and the root $\rho$. Here, each $T_i$ is an $x$-split. Thus the number of root-containing subtrees in $H$ is 
\[F_{H}(\rho) = (h+x^2+x+1)^{h+x-1}.\]
  Let $T'_i$ be an $(x-1)$-split.
Define a new tree $T'$ by replacing $T_i$ with $T'_i$ for each $i\in [h+x-1]$ and increasing the degree of the root by one so that the new branch $T'_0$ is also an $(x-1)$-split. 
Let $H'$ be the subtree induced by $T'_0, T'_1, \ldots, T'_{h+x-1}$ and the root of $T'$.
The number of root-containing subtrees in $H'$ is 
\[F_{H'}(\rho) = (h+x^2-x+1)^{h+x}.\]
In order to compare the number of root-containing subtrees of $T$ and $T'$, it suffices to compare the number of root-containing subtrees of $H$ and $H'$.

In order to compare these, consider the ratio: 
\begin{align*}
\frac{F_T(\rho)}{F_{T'}(\rho)} 
 = \frac{F_{H}(\rho)}{F_{H'}(\rho)}
 =& \frac{(h+x^2+x+1)^{h+x-1}}{(h+x^2-x+1)^{h+x}}\\
=& \frac{1}{h+x^2-x+1} \left(1+\frac{2x}{h+x^2-x+1}\right)^{h+x}\\
 \geq & \frac{1}{h+h^{2/t}+h^{1/t}+1} \left(1+\frac{2h^{1/(t+1)}}{h+h^{2/t}-h^{1/(t+1)}+1}\right)^{h+h^{1/(t+1)}} 
 \end{align*}
 since $x\in [h^{1/(t+1)}, h^{1/t}]$. The last expression can be further rewritten as 
\begin{align*} 
& \frac{1}{h+h^{2/t}+h^{1/t}+1} \left(1+\frac{2}{h^{t/(t+1)}+h^{(t+2)/(t^2+t)}-1+h^{-1/(t+1)}}\right)^{h+h^{1/(t+1)}}\\
 \geq & \frac{1}{3h} \left(1+\frac{2}{h^{t/(t+1)}+h^{(t+2)/(t^2+t)}}\right)^{h} \tag*{since $h^{1/(t+1)}>1$}\\
  \geq & \frac{1}{3h} \left(1+\frac{2}{h^{t/(t+1)}+h^{t/(t+1)}}\right)^{h^{t/(t+1)}h^{1/(t+1)}} \tag*{since $\frac{t+2}{t^2+t}< \frac{t}{t+1}$}\\
  = &   \frac{1}{3h} \left(1+\frac{1}{h^{t/(t+1)}}\right)^{h^{t/(t+1)}h^{1/(t+1)}}\\
\geq & \frac{1}{3h}\cdot \frac{1}{2} e^{h^{1/(t+1)}}   \tag*{since $\left( 1+\frac{1}{y} \right)^y > \frac{1}{2}e$ for $y\geq 1$}\\
> & 1 \tag*{since $h^{1/(t+1)} >\ln (6h)$.} 
\end{align*}
 Thus,  $T$ is not an optimal tree because $T'$ also has $n$ vertices and height $h$ but has fewer subtrees which contain its root. 
\end{proof}

As a consequence we obtain the following corollary:
\begin{cor}
Fix $h\geq 550$ and $n\in \mathbb{Z}$ with $n\geq 5h^2$. Let $T$ be an optimal tree with subtrees $T_i$ which are $k_i$-splits. Using the terminology from Theorem~\ref{thm:optimal_summary}, if $T$ has an ``even distribution'' of $k_i$ values, then $k_i\leq \ln(6h)$ for all $i\in \{1,2,\ldots, r\}$. If $T$ has ``one large'' $k_i$ values, then $k_i\leq \ln(6h)$ for all $i\in \{2,\ldots, r\}$.
\label{cor:sub9}
\end{cor}

\begin{proof}
Fix $n,h$ with $n\geq 5h^2$. Let $T$ be an optimal tree with height $h$ and $n$ vertices. Toward a contradiction, suppose there is an integer $x \in (\ln(6h), \sqrt{h+ \frac{5}{4}} - \frac{1}{2}]$ such that some $T_i$ is an $x$-split. Consider two cases. 

If $x \in \left(\ln(6h), {t^{1/3}}\right],$ then setting $t:=\frac{\ln h}{\ln x}-1$, we have $h^{1/(t+1)}=x>\ln(6h)$ and $t\geq 2$. Since $n\geq 5h^2$ and $x<h^{1/3}$, then clearly $n \geq (h+x)(h+x-1)+1$. Thus by Lemma~\ref{lem:technical_ki}, there are at most $h+x-1$ subtrees $T_i$ of $T$ which can be $x$-splits. 

If instead $x\in \left[h^{1/3},\sqrt{h+ \frac{5}{4}} - \frac{1}{2}\right]\subseteq \left[h^{1/3}, h^{1/2}\right]$ (where $h \geq 550$), then set $t:=2$. In this case, again $h^{1/(t+1)}> \ln (6h)$, $t\geq 2$, and $x\in [h^{1/(t+1)}, h^{1/t}]$ with $n \geq (h+x)(h+x-1)+1$. Therefore, by Lemma~\ref{lem:technical_ki}, there are at most $h+x-1$ subtrees $T_i$ of $T$ which can be $x$-splits. 

Now if $T$ has an ``even distribution'' of $k_i$ values, then $k_i \in \{x-1,x\}$ for some $x \leq \sqrt{h}$ for all $i\in \{1,2,\ldots, r\}$. Based on the conclusions here, the number of vertices in $T$ can be bounded as follows, which yields a contradiction:
\begin{align*}
 n &\leq (h+x)(h+x-1) + (h+x-1)(h+x-2) + 1\\
 &\leq (h+\sqrt{h})(h+\sqrt{h}-1)(h+\sqrt{h}-1)(h+\sqrt{h}-2) + 1 \\
 &< 5h^2.\end{align*}

On the other hand, if $T$ has ``one large'' $k_i$ value, then $k_1\leq h-1$ and there is an integer $x< \sqrt{h}$ such that $k_i=x$ for $i\in \{2,3,\ldots, r\}$. Therefore, the number of vertices in $T$ is bounded as follows, which produces a contradiction: 
$$n \leq (h+k_1) + (h+x)(h+x-1) + 1 \leq 2h + (h+\sqrt{h})(h+\sqrt{h}-1) + 1 <5h^2.$$
\end{proof}
Thus we have established Theorem~\ref{thm:optimal_summary}. 



\section{Concluding remarks}

In this study, we considered distances between three fundamental concepts of middle parts of the tree: the center, centroid, and subtree core. The maximum distances between each pair of these is first examined for general trees. The extremal structures that achieve these maximum distances contained a vertex of large degree and a long path, motivating us to study the same question for trees with degree restrictions and for trees with bounded diameter. The latter leads us to an interesting and difficult problem of minimizing number of root-containing subtrees among trees with given order and height. While we do not yet have a complete characterization of such optimal trees, we have established many of their structural properties to guide our continued study of this topic. 

\section{Acknowledgements}

Heather Smith was supported in part by the NSF DMS, contracts 1300547 and 1344199, as well as a Dean's Dissertation Fellowship awarded by the College of Arts and Sciences of the University of South Carolina. L\'aszl\'o Sz\'ekely was supported in part by the NSF DMS, contracts 1300547 and 1600811. Hua Wang was supported in part by the Simons Foundation, grant number 245307.


\begin{thebibliography}{10}

\bibitem{eric}
E.~Andriantiana, S.~Wagner, and H.~Wang.
\newblock Greedy trees, subtrees and antichains.
\newblock {\em Electron. J. Combin.}, 20:1--25, 2013.

\bibitem{barefoot}
C.A. Barefoot, R.C. Entringer, and L.A. Sz\'ekely.
\newblock Extremal values for ratios of distances in trees.
\newblock {\em Discrete Appl. Math.}, 80:37--56, 1997.

\bibitem{entringer}
R.~Entringer, D.~Jackson, and D.~Snyder.
\newblock Distance in graphs.
\newblock {\em Czechoslovak Math. J.}, 26:283--296, 1976.

\bibitem{jamison}
R.~Jamison.
\newblock On the average number of nodes in a subtree of a tree.
\newblock {\em J. Combinatorial Theory Ser. B}, 35:207--223, 1983.

\bibitem{jordan}
C.~Jordan.
\newblock Sur les assemblages de lignes.
\newblock {\em J. Reine Angew. Math.}, 70:185--190, 1869.

\bibitem{lovasz}
L.~Lov\'asz.
\newblock {\em Combinatorial Problems and Exercises.}
\newblock AMS Chelsea Publishing, Providence, Rhode Island, 2 edition, 2007.

\bibitem{nina}
N.~Schmuck, S~Wagner, and H.~Wang.
\newblock Greedy trees, caterpillars, and wiener-type graph invariants.
\newblock {\em MATCH Commun. Math. Comput. Chem.}, 68:273--292, 2012.

\bibitem{laszlo}
L.A. Sz\'ekely and H.~Wang.
\newblock On subtrees of trees.
\newblock {\em Adv. in Appl. Math.}, 34:138 -- 155, 2005.

\bibitem{laszlo2}
L.A. Sz\'ekely and H.~Wang.
\newblock Binary trees with the largest number of subtrees.
\newblock {\em Discrete Appl. Math.}, 155:374--385, 2007.

\bibitem{wang_dm}
H.~Wang.
\newblock The extremal values of the wiener index of a tree with given degree
  sequence.
\newblock {\em Discrete Appl. Math.}, 156:2647--2654, 2008.

\bibitem{wiener}
H.~Wiener.
\newblock Structural determination of paraffin boiling points.
\newblock {\em J. Am. Chem. Soc.}, 69:17--20, 1947.

\end{thebibliography}

\end{document}